\documentclass[11pt,a4paper]{article}
\usepackage{amsmath,amssymb,latexsym}
\usepackage{amsthm}
\usepackage{bm}
\usepackage{overpic}
\usepackage{color}
\theoremstyle{plain}
\newtheorem{theorem}{Theorem}[section]
\newtheorem{corollary}{Corollary}[section]
\newtheorem{proposition}{Proposition}[section]
\newtheorem{lemma}{Lemma}[section]
\theoremstyle{definition}

\theoremstyle{remark}
\newtheorem{remark}{Remark}[section]

\newcommand{\supp}{\mathop{\rm supp}}

\newcommand{\Arg}{{\rm Arg}}

\newcommand{\field}[1]{\mathbb{#1}}

\newcommand{\R}{\field{R}}

\newcommand{\N}{\field{N}}
\newcommand{\C}{\field{C}}

\renewcommand{\P}{\field{P}}

\renewcommand{\Re}{\mathop{\rm Re}}
\renewcommand{\Im}{\mathop{\rm Im}}

\newcommand{\dsty}{\displaystyle}

\def\XXint#1#2#3{{\setbox0=\hbox{$#1{#2#3}{\int}$}
\vcenter{\hbox{$#2#3$}}\kern-.5\wd0}}

\title{Equilibrium measures in the presence of weak rational external fields}

\author{Ram\'{o}n Orive and Joaqu\'{\i}n F. S\'{a}nchez--Lara}

\date{\today}

\begin{document}

\vspace{1cm} \maketitle

\begin{abstract}

In this paper equilibrium measures in the presence of external fields created by fixed charges are analyzed. These external fields are a particular case of the so-called rational external fields (in the sense that their derivatives are rational functions). Along with some general results, a thorough analysis of the particular case of two fixed negative charges (``attractors') is presented. As for the main tools used, this paper is a natural continuation of \cite{MOR2015}, where polynomial external fields were thoroughly studied, and \cite{OrSL2015}, where rational external fields with a polynomial part were considered. However, the absence of the polynomial part in the external fields analyzed in the current paper adds a considerable  difficulty to solve the problem and justifies its separated treatment; moreover, it is noteworthy to point out the simplicity and beauty of the results obtained.

\end{abstract}

\section{Introduction}
This paper is devoted to the study of equilibrium measures in the real axis in the presence of rational external fields created by fixed charges. These are external fields of the form:
\begin{equation}\label{fixedfield}
\varphi (x) = \sum_{j=1}^q\,\gamma_j\,\log\,|x-z_j|\,,\;\gamma_j\in \R\,,\;z_j\in \C \,,
\end{equation}
where for $\gamma_k>0\,,$ $z_k$ must lie on $\C \setminus \R\,,$ and it is assumed that $\dsty \sum_{j=1}^q\,\gamma_j\,=\,T\,>\,0\,.$ These conditions ensure that given any $t\in (0,T)$, there exists a measure $\lambda_t = \lambda_{t,\varphi}$, such that $\lambda_t (\R) = t$, with compact support $S_t \subset \R\,,$ uniquely determined by the equilibrium condition (see e.g. \cite{Saff:97})
\begin{equation}\label{equilibrium}
V^{\lambda_t}(x) + \varphi (x)\,\begin{cases} = c_t\,,\;& x\in S_t\,, \\ \geq c_t\,,\;& x\in \R\,, \end{cases}
\end{equation}
where $c_t$ is called the equilibrium constant and for a measure $\sigma\,,$ $\dsty V^{\sigma}(x) = - \int \,\log |x-s|\,d\sigma(s)\,.$ The measure $\lambda_t$ is called the equilibrium measure in the presence of $\varphi$ and minimizes the weighted energy $$I_{\varphi}(\sigma) = -\iint \log |x-z|\,d\sigma(x)\,d\sigma(z) + 2\int\varphi (x) d\sigma (x)$$ among all measures $\sigma$ supported in the real axis and such that $\sigma (\R) = t\,.$

External fields \eqref{fixedfield} are called rational since their derivatives are rational functions:
\begin{equation*}\label{derivrational}
\varphi' (x) = \sum_{j=1}^q\,\gamma_j\,\frac{x-\Re z_j }{(x-z_j)(x-\overline{z}_j)}\,,\;x\in \R.
\end{equation*}
In this sense, this paper completes the analysis started in \cite{OrSL2015}, where rational external fields of the form
\begin{equation}\label{generalrat}
\varphi(x) = P(x) + \sum_{j=1}^q\,\gamma_j\,\log\,|x-z_j|\,,\;\gamma_j\in \R\,,\;z_j\in \C \,,
\end{equation}
$P$ being a polynomial of even degree $2p$, with $p\geq 1$, were considered. There, a particular case was treated with detail: a generalized Gauss-Penner model for which $p=q=2$. When $p\geq 1$, the polynomial part makes the external field strong enough to be admissible for any $t\in (0,+\infty)$. Conversely, when the polynomial part is absent, the external field is weaker and it is admissible just for $t\in (0,T)$. This important difference is one of the reasons for studying these weaker rational external field in a separated paper.

Of course, it is also possible to deal with rational external fields \eqref{fixedfield} with some $z_j\in \R$ and the corresponding $\gamma_j > 0$; but in that case, the conductor where the equilibrium problem is posed cannot be the real axis. This is the situation, for instance, when the asymptotics of Jacobi polynomials with varying non--classical parameters are handled (see e.g. \cite{MFMGOr2000}, \cite{KuMF2004} and \cite{MFOr2005}). In this situation, the support of the equilibrium measure consists of a finite union of arcs and curves in the complex plane; these arcs/curves satisfy a kind of symmetry with respect to the external field (the so-called ``S-symmetry'' introduced by H. Stahl during the eighties, see \cite{Stahl:86}). In a similar fashion, when asymptotics of Laguerre polynomials with varying non--classical parameters are studied, a rational external field with a polynomial part \eqref{generalrat} takes place (see \cite{MFMGOr2001}, \cite{KuML2001}, \cite{KuML2004}, \cite{DMOr2011}, \cite{AMMT2014} and, from the viewpoint of the Gauss-Penner Random Matrix models, \cite{AlMAMR2015}). Other setting where it is feasible having $z_j\in \R$ and $\gamma_j > 0$ is when the conductor is a proper subset of the real axis not containing points $z_j$; in that case, hard edges at the endpoints of the conductor arise (see e.g. \cite{OrSL2015}, where the conductor $[0,\infty)$ is considered).

Rational external fields appear in a number of applications in approximation theory, for instance when dealing with the asymptotic distribution of zeros of Orthogonal or Heine-Stieltjes polynomials; in particular, the application of ``purely rational'' external fields of the form \eqref{fixedfield} to the asymptotics of Heine-Stieltjes polynomials will be recalled below with more detail. But there are also important applications in random matrix theory, for example in the study of Gauss-Penner type models. The rest of this section will be devoted to describing briefly these applications.

In the second section, equilibrium problems in the presence of external fields \eqref{fixedfield} are handled in general, studying some properties like the asymptotic behavior of the equilibrium measure when $t$ (the size of the equilibrium measure or, equivalently in other contexts, the ``time'' or the ``temperature'', \cite{BlEy2003}-\cite{Bleher/Its03}, \cite{KuML 00} and \cite{MOR2015}, among others) tends to $T$, as well as the evolution of this equilibrium measure when other parameters (the ``heights'', $\Im z_j$, or the ``masses'', $\gamma_j$) of the external field vary.

Finally, the particular case of a rational external field created by two fixed charges will be treated with detail in Section 3. In this case, the support of the equilibrium measure may consist of one or two intervals (``one-cut'' or ``two-cut'', respectively), and we are mainly interested in the evolution of this support when $t$ travels through the interval $(0,T)$. Our main result is Theorem 3.1 below, though other results necessary for its proof, presented in Sections 2 and 3, are also of interest themselves; these proofs are collected in Section 4 in order to make the paper more readable. Finally, the geometrical aspects of the solution of our main problem presented in Theorem 3.1 below are illustrated in the final appendix.

It is noteworthy to recall that, regarding the methodology used, this paper is a natural continuation of \cite{MOR2015}, where this ``dynamical'' approach was thoroughly carried out for the case of polynomial external fields.

\subsection{Generalized Lam\'{e} equations and Heine-Stieltjes Polynomials}

In a series of seminal papers (see \cite{St 1885}-\cite{St 1885d}), T. J. Stieltjes (1856-1894) provided an elegant model for the electrostatic interpretation of the zeros of classical families of orthogonal polynomials (Jacobi, Laguerre and Hermite) and polynomial solutions of certain linear differential equations (the so-called Heine-Stieltjes polynomials). Regarding the latter case, he considered the following scenario: We are given $a_k \in \R,k=0,\ldots,p\,,$ respect., with $a_0<a_1<\ldots<a_{p-1}<a_p$, and $p+1$ positive charges $\rho_k,k=0,\ldots,p\,,$ placed at $a_k,k=0,\ldots,p\,,$ respect. Then, suppose we have $n$ unit positive charges which can move freely through the real interval $[a_0,a_p]\,,$ and assume that a mutual interaction according the logarithmic potential takes place. Then, if the free charges are placed at points $x_1,\ldots,x_n\,\in \,[a_0,a_p]\,,$ the (discrete) energy of the system is given by:
\begin{equation}\label{discrete energy}
E(x_1,\ldots,x_n)\,=\,-\sum_{i<j}\ln|x_i-x_j|\,-\sum_{k=0}^p
\rho_k \sum_{i=1}^{n}\ln|a_k-x_i|\,.
\end{equation}
Then, following the same approach used for Jacobi polynomials, he
showed that for the equilibrium positions, $(x_1^*,\ldots,x_n^*)$, which minimize the energy \eqref{discrete energy},
the associated monic polynomials $\displaystyle
y(x)\,=\,\prod_{j=1}^n(x-x_j^*)$ (Heine-Stieltjes polynomials) are
solutions of the Lam\'{e} differential equation
\begin{equation}\label{Lame classic}
A(x)y''+B(x)y'+C(x)y=0 \,,
\end{equation}
\noindent where $\displaystyle A(x)=\prod_{k=0}^{p}(x-a_k)\,\in
\mathbb{P}_{p+1}\,$ and $\displaystyle B\in \mathbb{P}_{p}$ such that
$\frac{B(x)}{A(x)}\,=\,\sum_{k=0}^p\frac{\rho_k}{x-a_k}\,,$ for
some polynomial $\displaystyle C \in \mathbb{P}_{p-1}$ (Van Vleck
polynomial). Heine \cite{He 1878} and Stieltjes \cite{St 1885}
showed that there exist $\displaystyle \binom{n+p-1}{p-1}$
polynomials $C \in \mathbb{P}_{p-1}$ for which the Lam\'{e}
differential equation \eqref{Lame classic} has a unique polynomial solution
$\displaystyle y=y_n\,$ of exact degree $n$, with $n$ simple zeros
located in $\displaystyle (a_0,a_p)$ (see also \cite{Sz 75} for a more complete and comprehensive proof). It is easy to see that \eqref{Lame classic} recovers
the Jacobi case for $p=1$. For $p=2$, we have the so-called Heun
differential equation (see e.g. \cite{Ron 95}). In addition, it is noteworthy to mention that the zeros of Heine-Stieltjes polynomials are actually a particular case of the so-called \emph{weighted Fekete points} (see e.g. \cite{Saff:97}).
This electrostatic model was used in \cite{MFSa2002} to obtain
the asymptotics of Heine-Stieltjes polynomials when $n$, that is,
the number of free charges, increases to infinity.

During the last
century, several extensions of the model above have taken place
(see \cite{MaMFMG 07} and references therein), but in most cases
the positivity of the residues $\rho_k$ has not been dropped or, what is the
same, the presence of attractive fixed charges has not been allowed. However, in this sense it is necessary to point out some results by A.
Grunbaum \cite{Grum 98}-\cite{Grum 01}  and Dimitrov and Van
Assche \cite{DiVA 00}.

In a more recent paper \cite{OrGa2010}, the following equilibrium problem was considered (for a counterpart of this problem in the Unit Circle, see \cite{Gr 02} and \cite{MFMGOr 05}). Let $m,\,n\in
\mathbb{N}\,$ and consider $m$ prescribed negative charges
$-\omega_{k}<0\,,\,k=1,\ldots,m\,,$ placed, respectively, at points
$z_{k}\in \mathbb{C}\setminus \mathbb{R}\,,\,k=1,\ldots,m$. Thus,
if we denote by $x_{k}\in \mathbb{R}\,,\,k=1,\ldots,n,$ the
positions of $n$ free positive unit charges, the (logarithmic)
energy of the system is given by:
\begin{equation}\label{ilustenergy}
E(x_{1},\ldots,x_{n})=-\sum_{1\leq j<k\leq
n}\log|x_{k}-x_{j}|+\sum_{j=1}^{n} \sum_{k=1}^{m}
\omega_{k}\log|z_{k}-x_{j}|\,.
\end{equation}

Hereafter, let us denote by $\displaystyle s\,=\,\sum_{k=1}^{m}\omega_{k}\,,$ the total mass of the prescribed charges. Then, it was shown (see \cite[Theorem 1]{OrGa2010}) that if $s > n-1\,,$ the energy functional \eqref{ilustenergy} has a
global minimum in $\mathbb{R}^n\,.$ This minimum is attained at a
point $\displaystyle (x_{1}^*,\ldots,x_{n}^*)\in \mathbb{R}^{n}\,,$ where
$-\infty<x_{1}^*< \ldots <x_{n}^*<+\infty\,$. However, despite the classical Heine-Stieltjes setting, the global minimum does not need to be unique (see \cite[Section 2.2]{OrGa2010}). In addition, it was shown that each minimizer (Heine-Stieltjes polynomial), $\displaystyle y(x)\,=\,\prod_{j=1}^n(x-x_j^*)$, is solution of a generalized Lam\'{e} differential equation of the form \eqref{Lame classic}, where
$$A(x) = \,\prod_{k=1}^m(x-z_k)(x-\overline{z}_k)\,,\;h(x)=\prod_{k=1}^{m}\left((x-z_{k})(x-\overline{z}_{k})\right)^{\omega_{k}}\,,\,B(x) = -A(x)\,\frac{h'(x)}{h(x)}\,,$$
for some Van Vleck polynomial $C\in \P_{2m-2}\,.$

In \cite{OrGa2010} it was also considered the asymptotics of Heine-Stieltjes polynomials when both $n$ and $s = s(n)$ tend to $\infty$, in such a way that $\displaystyle \lim_{n\rightarrow\infty}\,\frac{s}{n} = \theta > 1\,.$ In the framework of differential operators, asymptotics when $n$, the degree of the Heine-Stieltjes polynomials, tend to $\infty$ are often known as ``semiclassical'' and the corresponding when $s$ tend to $\infty$ are called ``thermodynamical''; thus, in this case, there is a combination of both types of asymptotics (which is also known in random matrix models as a ``double scaling''). In this sense, for each $n$ denote $\displaystyle \nu_n \,=\,\frac{1}{n}\sum_{k=1}^{m(n)}\omega_{nk}\delta_{z_{nk}}$, which is an atomic measure such that $\nu_n (\R) = \frac{s(n)}{n}\,> \frac{n-1}{n}\,,$ and suppose that
$$
\nu_n \stackrel{*}{\longrightarrow} \nu \,,\, n\in \Lambda \subset
\mathbb{N}\,\,\text{and}\,\,n\rightarrow \infty \,,
$$
in the weak-* topology, for some measure $\nu$ of size $\displaystyle \theta = \lim_{n\rightarrow \infty}\,\frac{s}{n}\,> 1$, with
compact support in $\mathbb{C}\setminus \mathbb{R}$ and some
infinite subsequence $\Lambda \subset \mathbb{N}$. Now, suppose
that, for each $n\in \mathbb{N}\,,$ $\{x_{nj}^*:j=1,\ldots,n\}$ is
an equilibrium configuration (that is, a global minimum) for the discrete equilibrium problem
\eqref{ilustenergy}. Then, denoting by $\displaystyle
\mu_n\,=\,\frac{1}{n}\sum_{j=1}^n \delta_{x_{nj}^*}$, in \cite[Theorem 3]{OrGa2010} it was proved (following the same approach as in \cite[Th.
2]{MFSa2002} and \cite[Th. 3.2]{MFMGOr 05}) that $\displaystyle \mu_n \stackrel{*}{\longrightarrow} \mu\,,\, n\in
\Lambda \subset \mathbb{N}\,\,\text{and}\,\,n\rightarrow \infty$,
where $\mu$ is the equilibirum measure of $\mathbb{R}\,$ in the
external field $\varphi =-V(\nu,\cdot)$. That is, the unit
counting measures of zeros of Heine-Stieltjes polynomials
converge, in the weak-* star topology, to the equilibrium measure
in the external field due to the potential of the
negative charge $\nu \,$.

In the general case not much more can be said, but the situation is different if the limit measure $\nu$ is atomic. In \cite{OrGa2010} the authors dealt with this case; in particular, it was considered the case $\nu = \gamma_1 \delta_{z_1} + \gamma_2 \delta_{z_2}\,,$ with $z_1, z_2 \in \C \setminus \R\,$ and $\gamma_1 + \gamma_2 >1\,$. In this sense, the so-called ``totally symmetric'' case, i.e. when $z_2 = -\overline{z}_1$ and $\gamma_2 = \gamma_1\,,$ was studied in detail. One of the main goals of this paper is dealing with the general situation where the ``heights'', $\Im z_1, \Im z_2\,,$ and the ``charges'', $\gamma_1, \gamma_2\,,$ are arbitrary positive real numbers.

Finally, let us point out that it is also possible to consider sequences of critical configurations (relative extrema or saddle points), not necessarily global minima, of the discrete energy \eqref{ilustenergy} for $n\rightarrow \infty$ and $s/n \rightarrow \lambda >1$. The limit measures of such sequences will be the so-called (continuous) critical measures, a class of measures to which the equilibrium measure belongs. In Section 2, a little bit more will be said about these critical measures (see \cite{MFRa2011} for an extensive study of them).

\subsection{Applications of rational external fields to random matrix models}

It is well known that another important circle of applications of equilibrium problems in the presence of external fields deals with the Random Matrix models (see e.g. \cite{MOR2015} and the exhaustive bibliography therein).

This is an important theory within the mathematical physics and, more precisely, the statistical mechanics.

Specifically, let us consider the set of $N\times N$ Hermitian matrices
$$
\left\{ M=\left( M_{jk}\right)_{j,k=1}^N:\; M_{kj}=\overline{M_{jk}} \right\}
$$
as endowed with the  joint probability distribution
$$
d\nu_N(M)=\frac{1}{\widetilde Z_N}\, \exp\left( -  \mathrm{Tr} \, V(M)\right)\, dM, \qquad
dM=\prod_{j=1}^N dM_{jj}\prod_{j\not=k}^N d\Re M_{jk}d\Im M_{jk},
$$
where $V : \, \R \to \R$ is a given function such that
the integral in the definition of the normalizing constant
$$
\widetilde Z_N = \int \exp\left( -  \mathrm{Tr} \, V(M)\right)\, dM.
$$
converges. Then, it is well-known (see e.g. \cite{Mehta2004}) that $\nu_N$ induces a joint probability distribution $\mu_N$ on the  eigenvalues $\lambda_1<\dots <\lambda_N$ of these matrices, with the density
\begin{equation*}
\label{Eigdensity}
\mu_N' (\bm \lambda)  = \frac{1}{ Z_N}\, \prod_{i<j} (\lambda_i-\lambda_j)^2 \exp\left( -  \sum_{i=1}^N V(\lambda_i)\right) ,
\end{equation*}
where $\bm \lambda = (\lambda_1,\dots ,\lambda_N)$, and with the corresponding \emph{partition function}
$$
Z_N=\int_\R\dots\int_\R \, \prod_{i<j} (\lambda_i-\lambda_j)^2 \exp\left( -  \sum_{i=1}^N V(\lambda_i)\right) \, d\lambda_1\dots d\lambda_N.
$$
In this sense, the \emph{free energy} of this matrix model is defined as
\begin{equation*}
F_N=-\frac{1}{N^2}\log Z_N.
\end{equation*}
In the physical context it is very important to study the (thermodynamical) limit
\begin{equation*}
F_\infty=\lim_{N\to\infty}F_N
\end{equation*}
(the so-called infinite volume free energy). The existence of this limit
has been established  under very general conditions on $V$, see e.g.~\cite{Johansson}.

The fact that
$$
d\mu_N (\bm \lambda)  = \frac{1}{ Z_N}\,e^{-2H_N(\bm \lambda)}\,d\bm \lambda\,,
$$
with
\begin{equation*}
\begin{split}
H_N(\bm \lambda) & = \sum_{i< j} \log\frac{1}{|\lambda_i-\lambda_j|} +  \sum_{j=1}^N \frac{V(\lambda_j)}{2} \\
& = \frac{N^2}{2}\,\left[-\int\int_{x\neq y}\,\log |x-y|\,d\nu_N(x) d\nu_N(y) + \int V(x)\,d\nu_N(x)\right] \\
& = \frac{N^2}{2}\,I_{\varphi}(\nu_N)\,,\;\;\varphi = \frac{V}{2}\,,
\end{split}
\end{equation*}
means that the value of $F_\infty$ is related to the solution of a minimization problem for the weighted logarithmic energy. Therefore, the corresponding minimizer is the \emph{equilibrium measure} associated to the external field $\varphi = \frac{V}{2}$.

It has been particularly studied the case of polynomial potentials $V$, and more precisely, the situation when $V$ is a quartic polynomial (see e.g. the recent monograph by Wang \cite{Wangbook} or the papers \cite{AMM2010}, \cite{BlEy2003}, \cite{Bleher99}, \cite{BPS95}, \cite{KuML 00} and \cite{MOR2015}, among many others), paying special attention to the phase transitions. In \cite{OrSL2015}, general rational external fields of type \eqref{generalrat} are handled in connection with the generalized Gauss-Penner model considered in \cite{Kris2006}: a $1$-matrix model whose action is given by $\dsty V(M) = tr\,\left(M^4-\,\log (v+M^2)\right)\,,$ in order to get a computable toy-model for the gluon correlations in a baryon background. The dimensionless parameter $v>0$ stands for the ratio of quark mass to coupling constant, and the logarithmic term (responsible of the rational nature of the external field) encapsulates the effect of the $N$-quark baryon. On the other hand, the ``purely rational'' external fields considered in the current paper are connected with the so-called ``multi-Penner'' matrix model, with action given by $\displaystyle W(M) = \sum_{j=1}^N\,\mu_j\,\log (M-q_i)\,,$ which is of interest in Gauge Theory, as well as in Toda systems (see e.g. \cite{Eguchi} and \cite{Todastrings}).

\section{Rational external fields}
The subject of the present paper is the study of equilibrium measures in the presence of external fields of the form \eqref{fixedfield}, with $\gamma_j >0$ and $z_j \in \C \setminus \R$ for $j=1,\ldots,q$, where $$\sum_{j=1}^q\,\gamma_j\,=\,T\,>\,0\,.$$ In this sense, and also regarding the methodology used, it is a continuation of \cite{OrSL2015}, where general external fields containing a polynomial part of the form \eqref{generalrat} were considered. However, the absence of the polynomial part in the present external field represents a significant increase in the difficulty to solve the problem.

In \cite{MOR2015} and \cite{OrSL2015} it was shown how a combined use of two main ingredients provide a full description of the evolution of the support of the equilibrium measure when the size of the measure, $t$, grows from $0$ to $\infty$. These main tools are an algebraic equation for the Cauchy transform of the equilibrium measure and a dynamical system for the zeros of the density function of this measure, based on the Buyarov-Rakhmanov seminal result in \cite{Buyarov/Rakhmanov:99}.

Indeed, suppose that the rational external field is of the form \eqref{fixedfield}. Then, with respect to the first ingredient, as it was seen in \cite[Theorem 2.1]{OrSL2015} within a more general context, we have that the Cauchy transform of the equilibrium measure, $\dsty \widehat{\lambda_t}(x) = \int \, \frac{d\lambda_t (s)}{x-s}\,,$ satisfies the relation:
\begin{equation}\label{Cauchytr}
(-\widehat{\lambda_t} + \varphi')\,(z)\,=\,(T-t)\,\sqrt{R(z)}\,=\,(T-t)\,\frac{B(z)\,\sqrt{A(z)}}{D(z)}\,,\;z\in \C \setminus S_t\,,
\end{equation}
for some monic polynomials $\dsty A(z) = \prod_{j=1}^{2k}\,(z-a_j)$ and $B(z) = \prod_{j=1}^{2q-k-1}\,(z-b_j)$, with $a_1 < \ldots < a_{2k} \in \R\,,$ (and thus, $R$ is a rational function). Here and on the sequel, we denote $\dsty D(z) = \prod_{j=1}^q\,(z-z_j)(z-\overline{z}_j)\,,$ which is a polynomial of degree $2q$. In fact, \eqref{Cauchytr}, as well as the previous results in \cite[Theorem 2.2]{MOR2015} and \cite[Theorem 1.1]{OrSL2015}, hold in the more general context of critical measures.

In addition, \eqref{Cauchytr} provides an expression for the density of the equilibrium measure:
\begin{equation}\label{density}
\lambda'_t (x) = \,\frac{T-t}{\pi i}\,\sqrt{R(x)^+}\,=\,\frac{T-t}{\pi }\,\sqrt{|R(x)|}\;x\in \cup_{j=1}^k\,[a_{2j-1},a_{2j}]\,.
\end{equation}
Therefore, the zeros of $R$ are the main parameters of the equilibrium problem; the ones with odd multiplicity determine the support $S_t$. Equating the residues of both members in \eqref{Cauchytr} at $z_j\,,j=1,\ldots,q$, we obtain (taking real and imaginary parts) $2q$ nonlinear equations for the $2q+k-1$ zeros of $R$. However, only when the one-cut case occurs, that is, $k=1$, this system completely determines the zeros. When $k>1$, extra conditions are needed; indeed, in order to fulfil condition \eqref{equilibrium}, we have that
\begin{equation}\label{extraconds}
\int_{a_{2j}}^{a_{2j+1}} \,\sqrt{R(x)}\,dx = 0\,,\;j=1,\ldots,k-1\,,
\end{equation}
which means that $B$ must have an odd number of zeros (counting their multiplicities) on each gap $(a_{2j},a_{2j+1})\,,\;j=1,\ldots,k-1\,,$ of the support. Observe that it provides a bound for $k$, the number of intervals comprising the support $S_t$: namely, it holds $k\leq q$.

 Recall that for rational external fields of the form \eqref{generalrat} with $p\geq 1$, the residue at infinity always provide some simple equations helping us to find the value of the endpoints and the other zeros of the density function in a rather easy way, at least in the one-cut situation (see \cite{OrSL2015}); unfortunately, this does not work in our ``purely rational'' situation ($P\equiv 0$). Thus, in the current setting even the one-cut case is difficult to be explicitly solved, as well as finding the values of parameters where phase transitions occur. However, despite the complicated calculations needed for solving explicitly the problem of determining the support, it is worth to point out the simplicity and beauty of the results (see Theorem 3.1 below).	

The second ingredient is based on a  ``dynamical'' description of the support of the equilibrium measure $\lambda_t$ in the real axis in the presence of an external field, which was proposed by Buyarov and Rakh\-ma\-nov in  \cite{Buyarov/Rakhmanov:99}. This seminal result basically asserts that at a certain ``instant'' $t_0$, the derivative of the equilibrium measure $\lambda_t$ with respect to $t$ is given by the Robin measure (that is, the equilibrium measure in the absence of an external field) for the support $S_{t_0}$. In the current rational case, taking into account these results and the well-known expression for the Robin measure of a finite union of compact intervals, we have that except for an at most denumerable set of values of $t$,
\begin{equation}\label{BuyRakhrat}
\frac{\partial}{\partial t}\,\left((T-t)\,\frac{B(z)\,\sqrt{A(z)}}{D(z)}\right)\,=\,-\,\frac{F(z)}{\sqrt{A(z)}}\,,
\end{equation}
where $F$ is a monic polynomial of degree $k-1$ such that $\dsty \int_{a_{2j}}^{a_{2j+1}}\,\frac{F(x)}{\sqrt{A(x)}}\,dx\,=\,0\,,\;j=1,\ldots,k-1\,,$ which means that $F$ has a simple zero on each gap $(a_{2j},a_{2j+1})$ of the support (see \cite{OrSL2015}). As in \cite{MOR2015}, we will make use of the abbreviate physical notation for the derivative with respect to the ``time'' $t$: $\dsty \dot{f} = \frac{\partial f}{\partial t}\,.$ Thus, using \eqref{BuyRakhrat}, one immediately obtains (see \cite[Theorem 1.2]{OrSL2015}):

\begin{theorem}\label{dynamicalsystem}
Except for a denumerable set of values of $t$, it holds
\begin{equation}\label{variationzeros}
\begin{split}
\dot{a_i} & = \,\frac{1}{T-t}\,\frac{2D(a_i)F(a_i)}{B(a_i)\,\prod_{j\neq i}(a_i-a_j)}\,,\,i=1,\ldots,2k \,,\\
\dot{b_i} & = \,\frac{1}{T-t}\,\frac{D(b_i)F(b_i)}{A(b_i)\,\prod_{j\neq i}(b_i-b_j)}\,,\,i=1,\ldots,2q-k-1\,,
\end{split}
\end{equation}
with $1\leq k\leq q\,.$
\end{theorem}
From \eqref{variationzeros}, it is clear that always $\dot{a_1}<0$ and $\dot{a_{2k}}>0\,.$ Moreover, taking into account that on each gap there is an even number of zeros of the product $B(x)F(x)$, it is also possible to assert that the $a_i$ when $i$ is odd are decreasing, while for even values of $i$ are increasing, which is coherent with the well-known fact that the support $S_t$ is increasing with $t$ (see \cite{Buyarov/Rakhmanov:99}).

Indeed, \eqref{variationzeros} is a dynamical system for the positions of all the important points determining the equilibrium measure and its support.
Previously, in \cite{MOR2015} a similar result was extensively used to study the dynamics of the equilibrium measure and its support, specially for the case of polynomial external fields; in particular, the so-called ``quartic'' case was analyzed in detail. Similarly, in \cite{OrSL2015} the case of a rational external field consisting of a polynomial part plus a logarithmic term (a generalized Gauss-Penner model) was studied.
\begin{remark}\label{lem:dynsystgral}
Bearing in mind the results in \cite{MOR2015}, for polynomial external fields, and those in \cite{OrSL2015} and \eqref{variationzeros}, for the rational case, it is easy to find a general structure of these dynamical systems. Indeed, for the zeros of the density \eqref{density} of the equilibrium measure of the interval $[c,d]\subset \R$ (bounded or not) in the presence of the external field \eqref{generalrat}, the following system of differential equations holds (except for a finite number of bifurcations/collisions),
\begin{equation}\label{dynsystgral}
\dot{\xi}_j = h_j(t)\,\frac{D(\xi_j)F(\xi_j)}{(AB)'(\xi_j)}\,,
\end{equation}
where $h_j(t)$ is a positive function of $t$ which reduces to a constant if the external field \eqref{generalrat} has a polynomial part, $D$ is real polynomial of even degree whose zeros are located at the point masses and their conjugates, and $A,B$ play the same role as in \eqref{Cauchytr}-\eqref{density}. At first sight, there seems to be an important difference between the purely rational case handled in the current paper and the rational cases with a polynomial part: in the present case, the system  is not an autonomous one, since function $h_j$ depends on the variable $t$; in the other rational cases, the presence of the polynomial part in the external field makes $h_j \equiv \kappa_j\,$ i.e constants independent of $t$. However, this is only an apparent difference: after a simple change of variable, the system easily becomes autonomous. Indeed, it is easy to check that the change $u = -\log (T-t) + \log T\,,$ with the new ``time'' $u$ lying on $(0,+\infty)$ transforms \eqref{variationzeros} in an autonomous system.

The shape of these dynamical systems \eqref{dynsystgral} resembles in a certain sense to a system of ODEs studied by Dubrovin in \cite{Dubrovin:1975fk} for the dynamics of the Korteweg-de Vries equation in the class of finite-zone or finite-band potentials, as it was pointed out in \cite{MOR2015}.

\end{remark}

As it was said above, under mild conditions on $\varphi$, the equilibrium measure depends analytically on $t$ except for a (possible) small set of values, which are called the critical points or the singularities of the problem. At this critical values of $t$, the so-called \emph{phase transitions} occur; in most of them, it implies a change in the number of cuts (connected components of the support $S_t$), but not always. The study of these phase transitions is one of the main issues of this problem. Let us recall, briefly, the basic type of singularities we can find (using the classification in \cite{MR2001g:42050}, also used then in \cite{KuML 00}) and \cite{MOR2015}). In this case, we prefer recalling the version of these definitions used in \cite{MOR2015}, namely:

\begin{itemize}
\item \textbf{Singularity of type I:} at a time $t=\tau$ a real zero $b$ of $B$ is such that $(V^{\lambda_{\tau}} + \varphi) (b)=c_{\tau}$, $b\notin S_{\tau}$ (see \eqref{equilibrium}), in such a way that for $t=\tau$, $b$ is a simple zero of $B$. Therefore, at this time $t=\tau$ a real zero $b$ of $B$ (a double zero of $R_{\tau}$) splits into two simple zeros $a_-<a_+$, and the interval $[a_-,a_+]$ becomes part of $S_t$ for $t>\tau$ (\emph{birth of a cut}).
 \item \textbf{Singularity of type II:} at a time $t=\tau$, a real zero $b$ of $B$ (of even multiplicity) belongs to the interior of the support $S_{\tau}$; according to \eqref{Cauchytr}-\eqref{density}, the density of $\lambda_{\tau}$ vanishes in the interior of its support, in such a way that for $t=\tau$, $b$ is a double zero of $B$. Thus, at this time $t=\tau$ two simple zeros $a_{2s}$ and $a_{2s+1}$ of $A$ (simple zeros of $R_t$, i.e., endpoints of the support) collide (\emph{fusion of two cuts}).
 \item \textbf{Singularity of type III:} at a time $t=\tau$, polynomials $A$ and $B$ have a common real zero $a$; the only additional assumption is that   $a$ is a double zero of $B$, so that $\lambda_{\tau}'(x)=\mathcal O(|x-a|^{5/2})$ as $x\to a$. Then, at this time $t=\tau$ a pair of complex conjugate zeros $b$ and $\overline b$ of $B$ (double zeros of $R_t$) collide with a simple zero $a$ of $A$ (endpoint), so that $\lambda_{\tau}'(x)=\mathcal O(|x-a|^{5/2})$ as $x\to a$. Observe that in this case no topological change takes place: the number of cuts does not vary.
 \end{itemize}

Finally, in \cite{MOR2015} and \cite{OrSL2015} it is also considered another special situation, which is not properly a singularity. It takes place when polynomial $B$ has two conjugate imaginary roots, $b$ and $\overline{b}$, which collide at a certain time and give birth a double real root of $B$ (quadruple real root of $R$) in the real axis, which immediately splits into a pair of double roots, $b_1$ and $b_2$, which tend to move away each other. In fact, what we have in this case is the birth of two new local extrema of the total (or ``chemical'') potential \eqref{equilibrium}. From this point of view, a type III singularity may be seen as a limit case of these situations, when the pair of imaginary zeros of $B$ collide simultaneously with a zero of $A$ (endpoint).

In the case where the number of cuts is bounded by $2$ (precisely, the case we will deal in Section 3 of this paper), these are just all the possible types of singularities, while where it can be greater than $2$ more intriguing phenomena can occur when two or more of these singularities take place simultaneously. 

Now, we are dealing with what may be called, in a colloquial style, ``the beginning and the end of the movie'', that is, the situation when $t\searrow 0$ and $t\nearrow T$. The answer to the first question is clear: since $\varphi'$ is a rational function whose numerator has degree $2N-1$, we easily conclude that the cardinality of the set $$\bigcap_{t>0}\,S_t = \{y\in \R\,:\,\varphi (y) = \min_{x\in \R} \varphi (x)\}$$ belongs to the set $\{1,\ldots,q\}\subset \N\,.$ Indeed, for $t=0$, the left-hand member \eqref {Cauchytr} reduces to $\varphi'$ and, thus, the support of the equilibrium measure starts at one or several of the critical points of the external field; in a similar fashion, these critical points of $\varphi$ are the initial conditions of the dynamical system \eqref{variationzeros}. Regarding the second question, we have,

\begin{theorem}\label{thm:end}

Denote by $\mu_t$ the equilibrium measure in the external field \eqref{fixedfield}. Then,
\begin{itemize}
\item[i)] There exists the limit of the equilibrium measure when $t\rightarrow T$:
$$\lim_{t\rightarrow T}\mu_{t}=\mu_{T}\,,$$
in the sense that $$\lim_{t\rightarrow T}\mu_t (I)=\mu_T (I)\,,$$
for any Borel set $I\subset\mathbb{R}$.

\item[ii)] For $t$ sufficiently close to $T$, the support $S_t$ consists of a single interval. In particular, using the $AB$-representation \eqref{Cauchytr}, the zeros of polynomial $A$ (endpoints) diverge:
    $$\lim_{t\rightarrow T} a_1=-\infty\,,\qquad \lim_{t\rightarrow T} a_2=+\infty\,,$$
    and the zeros of polynomial $B$ converge to the zeros of the rational function:
    $$\sum_{j=1}^q\left(\frac{\gamma_j}{2(z-z_j)}-\frac{\gamma_j}{2(z-\overline{z}_j)}\right)\,,$$
    which belong to $\mathbb{C}\setminus\mathbb{R}$.
\item[iii)] The density of the limit measure $\mu_T$ is given by
$$\frac{d\mu_T}{dx}=\,\frac{1}{T\,\pi}\,\sum_{j=1}^q\frac{\gamma_j\Im z_j}{((x-\Re z_j)^2+\Im z_j^2)}$$
\end{itemize}

\end{theorem}

To end this section, let us consider the variation of the equilibrium measure when some of the parameters of the external field \eqref{fixedfield} vary; that is, we mean the evolution of the equilibrium measure when one of the ``masses'' $\gamma_j$ or ``heights'' $\beta_j$ is varying.

In order to do it, we present now a simplified version of \cite[Theorem 5]{MOR2015}, where the authors extended the seminal Buyarov-Rakhmanov result for the variation of the equilibrium measure with respect to other parameters in the external field, which is sufficient for our purposes (see also \cite[Theorem 1.3]{OrSL2015})
\begin{theorem}\label{parameters}
Let $t>0$ be fixed and suppose that the function $\varphi(x;\tau)$ is real-analytic for $x\in \R$ and $\tau \in (c,d)$, where $(c,d)$ is a real interval. Let $\lambda = \lambda_{t,\tau}$ denote the equilibrium measure in the external field $\varphi(x;\tau)$, for $\tau \in (c,d)$, with support $S_{t,\tau}$. Then, for any $\tau_0 \in (c,d)$, $$\frac{\partial \lambda}{\partial \tau}\,|_{\tau=\tau_0} = \omega\,,$$ where the measure $\omega$ is uniquely determined by the conditions
\begin{equation}\label{condparam}
\supp \omega = S_{t,\tau_0}\,,\;\omega (S_{t,\tau_0}) = 0\,,\;V^{\omega} + \frac{\partial \varphi(x;\tau)}{\partial \tau}\,|_{\tau=\tau_0} = \frac{\partial c_t}{\partial \tau}\,|_{\tau=\tau_0} = const\;\;\text{on}\;\; S_{t,\tau_0}
\end{equation}

\end{theorem}

Observe that the second formula in \eqref{condparam} means that $\omega$ is a type of signed measure which is often called a \emph{neutral} measure; it is a natural consequence of the fact that $t$, the total mass of $\lambda$, does not depend on parameter $\tau$.

We are concerned, first, with the situation when one of the ``masses'' $\gamma_j$, with $j\in\{1,\ldots,q\}\,,$ varies in $(0,+\infty)$. In this case, taking into account that $\displaystyle \frac{\partial \varphi(z)}{\partial \gamma_j} = \log |z-z_j|\,,$ Theorem \ref{parameters} above and the $AB$-representation \eqref{Cauchytr}-\eqref{density} yield
\begin{equation}\label{parammass}
\frac{\partial}{\partial \gamma_j}\,\left((T-t)\,\frac{B(z)\sqrt{A(z)}\,}{D(z)}\right)\,=\,\frac{H(z)}{(z-z_j)(z-\overline{z_j})\,\sqrt{A(z)}}\,,
\end{equation}
where $H$ is a monic polynomial of degree $(k+1)$ which has a zero of odd order on each of the $(k-1)$ gaps of the support of the equilibrium measure, and such that at least one of the other two zeros lies inside the convex hull of the support. In particular, when $k=1$, that is, the one-cut case, \eqref{parammass} yields, for the endpoints of the single interval comprising the support, the following dynamical system:
\begin{equation*}\label{dynammass}
\begin{split}
\frac{\partial a_1}{\partial \gamma_j} = & \,-\,\frac{2}{T-t}\,\frac{(a_1-h_1)(a_1-h_2)\,\widetilde{D_j}(a_1)}{(a_1-a_2)\,B(a_1)}\,, \\
\frac{\partial a_2}{\partial \gamma_j} = & \,-\,\frac{2}{T-t}\,\frac{(a_2-h_1)(a_2-h_2)\,\widetilde{D_j}(a_2)}{(a_2-a_1)\,B(a_2)}\,,
\end{split}
\end{equation*}
where $B$ is, in this case, a monic polynomial of degree $2q-2$ being positive in the interval $(a_1,a_2)$, $h_1,h_2$, the roots of polynomial $H$ in \eqref{parammass}, is a pair or real numbers, such that at least one of them belongs to $(a_1,a_2)$, and
$$ \widetilde{D_j}(z) = \,\frac{D(z)}{(z-z_j)(z-\overline{z_j})}\,=\,\prod_{l\neq j}\,(z-z_l)(z-\overline{z_l})\,=\,\prod_{l\neq j}\,\left((z-\Re z_l)^2+(\Im z_l)^2\right)\,>\,0\,.$$
Therefore, the increase or decrease of the endpoints when $\gamma_j$ grows depends on the position of the points $h_1,h_2\,,$ which, in turn, is determined by the relative position of the charge $z_j$ in the set $\{z_1,\ldots,z_q\}$. In a similar way, the dynamical system for the other zeros of the density of the equilibrium measure ($2q-2$ zeros of polynomial $B$) may be displayed.

In a similar fashion, the evolution of the support when one of the heights $\beta_j$ varies may be handled. In this case, we have that $\displaystyle \frac{\partial \varphi(z)}{\partial \beta_j} = \frac{\gamma_j\,\beta_j}{(z-z_j)(z-\overline{z_j})}\,,$ and thus, Theorem \ref{parameters} implies that
\begin{equation}\label{paramheight}
\frac{\partial}{\partial \beta_j}\,\left((T-t)\,\frac{B(z)\,\sqrt{A(z)}}{D(z)}\right)\,=\,\frac{K(z)}{(z-z_j)^2\,(z-\overline{z_j})^2\,\sqrt{A(z)}}\,,
\end{equation}
where now the polynomial $K$, not necessarily monic, has degree $\leq (k+2)$ and a zero of odd multiplicity on each of the $(k-1)$ gaps of the support, and such that at least one of the other zeros lies inside the convex hull of the support.
As above, for $k=1$, that is, when the one-cut case takes place, \eqref{paramheight} yields the following dynamical system for the endpoints of the single interval comprising the support:
\begin{equation*}\label{dynamheight}
\begin{split}
\frac{\partial a_1}{\partial \beta_j} = \,& \frac{2}{T-t}\,\frac{K(a_1) D(a_1)}{(a_1-a_2)\,B(a_1)\,(a_1-z_j)^2\,(a_1-\overline{z}_j)^2}\,, \\
\frac{\partial a_2}{\partial \beta_j} = \,& \frac{2}{T-t}\,\frac{K(a_2) D(a_2)}{(a_2-a_1)\,B(a_2)\,(a_2-z_j)^2\,(a_2-\overline{z}_j)^2}\,,
\end{split}
\end{equation*}

The next section is devoted to the simplest (but quite difficult) non-trivial case, where the external field is created by two prescribed charges, that is, $q=2$. The full description of the dynamics of the equilibrium measure may be done for this situation.

\section{An external field created by two prescribed charges}

Throughout this Section, we restrict to the case of the equilibrium problem in the presence of a couple of (attractive) prescribed charges. In particular, and without lack of generality, we consider external fields of the form:
\begin{equation}\label{couple}
\varphi (x) = \log |x-z_1| + \gamma \,\log \, |x-z_2|\,,\;\gamma > 0\,,\;z_1,z_2 \in \C \setminus \R\,.
\end{equation}
That is, we are concerned with the case where $q=2$ in \eqref{fixedfield} and, thus, we know that the number of intervals (``cuts'') comprising the support $S_t$ is given by $1$ or $2$. We can assume $\Re z_1 = -\Re z_2 = -1$, as well as $\Im z_1 = \beta_1 > 0, \Im z_2 = \beta_2 >0\,,$ also without loss of generality.  Now, the evolution of the equilibrium measure $\lambda_t$ and, in particular, of its support $S_t$, in the presence of the external field \eqref{couple}, depending of three parameters, $\beta_1,\beta_2 > 0$ and $\gamma > 0$, for $t\in(0,T)$, with $T = 1+\gamma$, is investigated.

In the particular case handled in this section, where the external field is due to a couple of prescribed charges \eqref{couple}, we have that \eqref{Cauchytr} holds, with $deg A \in \{2,4\}$ and $deg B = 3- \frac{deg A}{2}$. Thus, taking residues at $z=z_1$ and $z=z_2$ in \eqref{Cauchytr} yields
\begin{equation}\label{residues}
\begin{cases}(T-t)\, B(z_1)\,\sqrt{A(z_1)}\,-\,i\,\Im z_1\, (z_1-z_2)(z_1-\overline{z}_2) & = 0\,,\\
(T-t)\,B(z_2)\,\sqrt{A(z_2)}\,-\,i\,\gamma \,\Im z_2\, (z_2-z_1)(z_2-\overline{z}_1) & = 0\,, \end{cases}
\end{equation}
and, after taking real and imaginary parts, we finally arrive to a nonlinear system of four equations. Thus, system \eqref{residues} determines uniquely polynomials $A$ and $B$ in the one--cut case; but if the support consists of two disjoint intervals, then an additional condition \eqref{extraconds} is also necessary.

Now, combining the two ingredients above, that is, formulas \eqref{Cauchytr} and \eqref{variationzeros}, we have the following possible settings for the support $S_t$ of the equilibrium measure and its density (for non-singular values of $t\in (0,T)$).
\begin{itemize}

\item [(\textbf{one-cut})] $S_t = [a_1,a_2]\,,\;a_1 = a_1(t) < a_2 = a_2(t)\,$ and
\begin{equation*}\label{density1}
\lambda'_t (x) = \,\frac{T-t}{\pi}\,\frac{(x-b_1)(x-b_2)\,\sqrt{(x-a_1)(a_2-x)}}{D(x)}\,,
\end{equation*}
with $b_2<b_1<a_1<a_2$ and
$$ \int_{b_2}^{a_1}\,\frac{(x-b_1)(x-b_2)\,\sqrt{(x-a_1)(x-a_2)}}{D(x)}\,dx\,>\,0\,,$$
 or $a_1<a_2<b_1<b_2$ and
 $$ \int_{a_1}^{b_2}\,\frac{(x-b_1)(x-b_2)\,\sqrt{(x-a_1)(x-a_2)}}{D(x)}\,dx\,>\,0\,,$$ or, finally, $b_2 = \overline{b}_1 \in \C \setminus \R\,.$ In this scenario, we also have, for $i,j=1,2$ and $j\neq i$,
\begin{equation}\label{dynamics1}
\dot{a_i} = \frac{1}{T-t}\,\frac{2D(a_i)}{(a_i-a_j)(a_i-b_1)(a_i-b_2)}\,,\;\dot{b_i} = \frac{1}{T-t}\,\frac{D(b_i)}{(b_i-b_j)(b_i-a_1)(b_i-a_2)}\,.
\end{equation}

\item [(\textbf{two-cut})] $S_t = [a_1,a_2]\,\cup \,[a_3,a_4]\,,\;a_1 = a_1(t) < a_2 = a_2(t) < a_3 = a_3(t) < a_4 = a_4(t)\,$ and
\begin{equation*}\label{density2}
\lambda'_t (x) =
\,\frac{T-t}{\pi}\,\frac{|x-b_1|\,\sqrt{(x-a_1)(x-a_2)(x-a_3)(a_4-x)}}{D(x)}\,,
\end{equation*}
with $a_2<b_1<a_3$ and $$\int_{a_2}^{a_3}\,\frac{(x-b_1)\sqrt{(x-a_1)(x-a_2)(x-a_3)(x-a_4)}}{D(x)}\,dx\,=\,0\,.$$ In this case, it holds the dynamical system
\begin{equation*}\label{dynamics2}
\dot{a_i} = \frac{1}{T-t}\,\frac{2D(a_i)\,F(a_i)}{A'(a_i)(a_i-b_1)}\,,\;\dot{b_1} = \frac{1}{T-t}\,\frac{D(b_1)\,F(b_1)}{A(b_1)}\,,\;i=1,2,3,4\,,
\end{equation*}
where $F(x) = x - \zeta$, with $\zeta$ uniquely determined by the condition $$\int_{a_2}^{a_3}\,\frac{(x-\zeta)}{\sqrt{(x-a_1)(x-a_2)(x-a_3)(x-a_4)}}\,dx\,=\,0\,.$$
\end{itemize}

Now, we are going to state our main result. In order to do it, consider the bivariate polynomial
\begin{equation}\label{boundary}
f(x,y) = 27xy(x-y)^2 - 4(x^3+y^3) + 204xy(x+y) - 48(x^2-7xy+y^2+4x+4y) - 256\,.
\end{equation}

For $x,y >0$, it is a symmetric function with respect to its arguments, and the graphic of the curve
\begin{equation}\label{curveC}
\mathcal{C} = \left\{(\beta_1,\beta_2) \in (0,+\infty) \times (0,+\infty)\,:\,f(\beta_1^2,\beta_2^2) = 0 \right\}
\end{equation}
is decreasing and splits the open first quadrant of the $(\beta_1,\beta_2)$-plane into two domains: $\Omega_0$, with the origin belonging to its closure, and $\Omega_{\infty}$ (see Figure \ref{fig:regionesomegas}). The curve $\mathcal{C}$ has two asymptotes at $\beta_1 = \,\frac{2}{3\,\sqrt{3}}$ and $\beta_2 = \,\frac{2}{3\,\sqrt{3}}\,.$

\begin{figure}
    \begin{center}
        \includegraphics[scale=0.6]{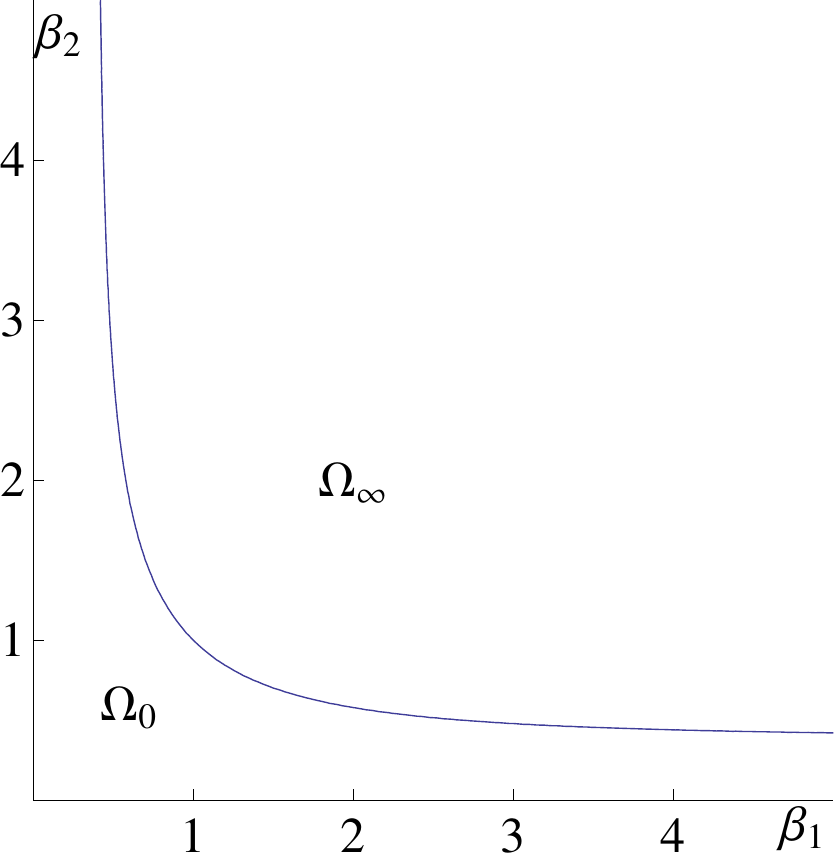}
    \end{center}
    \caption{Regions $\Omega_0$ and $\Omega_{\infty}$ and the curve $\mathcal{C}$ }
    \label{fig:regionesomegas}
\end{figure}

\begin{theorem}\label{thm:main}
Let $S_t$ be the support of $\lambda_t$, the equilibrium measure in the external field \eqref{couple}, with $z_1 = -1+i \beta_1, z_2 = 1+i \beta_2$. Then, we have

\begin{itemize}

\item If $(\beta_1,\beta_2)\in \Omega_{\infty}\,\cup \,\mathcal{C}$ (that is, $f(\beta_1^2,\beta_2^2) \geq 0$), then $S_t$ consists of a single interval (``one-cut'') for any $\gamma > 0$ and $t\in (0,T)$.

\item If $(\beta_1,\beta_2)\in \Omega_0$ ($f(\beta_1^2,\beta_2^2) < 0$), then there exist two values $0 < \Gamma_1 = \Gamma_1 (\beta_1,\beta_2) < \Gamma_2 = \Gamma_2 (\beta_1,\beta_2)$ such that for $\gamma \in (\Gamma_1, \Gamma_2)$ there are two critical values $0 \leq T_1 = T_1(\beta_1,\beta_2,\gamma) < T_2 = T_2(\beta_1,\beta_2,\gamma) < T$, in such a way that $S_t$ consists of two disjoint intervals for $t\in (T_1,T_2)\,$ (``two-cut''). Otherwise, $S_t$ consists of a single interval.

\end{itemize}

\end{theorem}

\begin{remark}\label{rem:recipe}
The expression of the ``boundary-curve'' $\mathcal{C}$ may be easily obtained by imposing that the derivative $\varphi'(x)$ of the external field \eqref{couple} has a triple real root. The recipe to determine the values of $\Gamma_1,\Gamma_2$ and $T_1,T_2$ will be shown within the proofs in Section 4 below.
\end{remark}

\begin{remark}\label{rem:heights}
The result in Theorem \ref{thm:main} means that the relationship between the distances of the two attractive charges to the real axis (``heights'') determines the possible existence of a two-cut phase. It seems natural, but a curious phenomenon also takes place: if one of the charges is close enough to the real axis, say $\beta_1 < \frac{2\sqrt{3}}{9} \sim 0.385$, a range of admissible values of the ``mass'' $\gamma$ may be found for any value of the other height, $\beta_2$, in order to allow the existence of a two--cut phase. Roughly speaking, it seems to tell us that if one of the charges is sufficiently close to the real axis, then it is always possible to distinguish both charges from there, whatever the distance of the other one (provided a suitable fit between the masses, of course).

Obviously, this ``positive'' result has a ``negative'' counterpart: if the couple of attractive charges are sufficiently far from the real axis (i.e., $(\beta_1,\beta_2)\in \Omega_{\infty}\cup \emph{C}$), they are indistinguishable from there (in the sense that they are unable to split the support of the equilibrium measure) whatever the masses.

\end{remark}

\begin{remark}\label{rem:end}
The result in previous Theorem \ref{thm:end} may be easily illustrated in this case. Indeed, we have that the density of the limit measure (as $t\rightarrow T=1+\gamma$) is given by
$$\frac{d\mu_T}{dx}\,=\,\frac{1}{T \pi}\,\left(\frac{\beta_1}{(x+1)^2+\beta_1^2}\,+\,\frac{\gamma \,\beta_2}{(x-1)^2+\beta_2^2}\right)\,,\;x\in \R\,,$$
whose zeros are imaginary for any $\beta_1, \beta_2, \gamma > 0$.

\end{remark}

\begin{remark}\label{rem:particular}
As it was said above, some particular situations were considered in \cite{OrGa2010}; in particular, the so-called ``totally symmetric'' case, that is, where heights and masses are equal ($\beta_2 = \beta_1 = \beta$ and $\gamma = 1$ in our current notation) and a ``partially symmetric'' case, where just the heights are supposed to be equal. With respect to the latter one, it is easy to see that the intersection between the curve $\mathcal{C}$ and the bisector $\beta_2 = \beta_1$ consists of the point $\beta_1 = \beta_2 = 1$. Thus, we conclude that a two-cut phase is feasible in this partially symmetric situation if and only if the common height $\beta < 1$ or, what is the same, if the two charges are close enough to the real axis to be able to split the support.

The special situation of the totally symmetric case will be revisited with more detail after Theorem 3.4 below, where the evolution of the support $S_t$ is described.
\end{remark}

\begin{remark}\label{exp convex}
It is well-known that the convexity of the external field ensures that the support of the equilibrium measure is an interval (see e.g. \cite{Saff:97}). In \cite{BDD2006} a weaker sufficient condition is given, namely, the convexity of the function $\exp (\varphi)\,.$ We can check whether this condition is fulfilled when $(\beta_1,\beta_2)\in \Omega_{\infty} \,\cup \,\mathcal{C}$ and, thus, whether in this sense the first part of Theorem \ref{thm:main} is a consequence of that previous result. However, it is possible to find examples with $(\beta_1,\beta_2) \in \Omega_{\infty} \,\cup \,\mathcal{C}$ and $\gamma >0$ such that $\exp \varphi$ is not convex. Indeed, it is easy to check that $(\beta_1,\beta_2)\in \Omega_{\infty}$ for $\beta_1 = 0.5$ and $\beta_2 = 2.7$, but $\exp \varphi$ is non-convex for these values when we take, for instance, $\gamma = 5.6$.

\end{remark}

\begin{remark}\label{rem:OP}
In a similar fashion as in the applications to the asymptotics of Heine-Stieltjes polynomials or to Random Matrix models considered in Section 1.2, results in Theorem \ref{thm:main} also may be used to describe the support of the limit zero distribution of polynomials $P_n\,,$ with $deg \, P_n = n$, satisfying varying orthogonality relations of the form $$\int \,x^k\,P_n(x)\,\omega_n(x)\,dx\,=\,0\,,\;k=0,\ldots,n-1\,,$$ where the varying weight $\omega_n$ is a generalized Jacobi--type weight given by
$$\omega_n(x) = \,\frac{1}{|x-z_1|^{\alpha_n}\,|x-z_2|^{\beta_n}}\,,\;x\in \R\,,$$ with $\alpha_n + \beta_n > 2n\,,\;n\in \N\,,$ in such a way that $\displaystyle \lim_{n\rightarrow \infty}\,\frac{\alpha_n}{n} = \,\emph{A}>0$ and $\displaystyle \lim_{n\rightarrow \infty}\,\frac{\beta_n}{n} = \,\emph{B}\,,$ and $A+B>2.$ Indeed, it is enough to set $\displaystyle \gamma = \, \frac{\emph{B}}{\emph{A}}$ and $t = \,\frac{2}{\emph{A}}$.

\end{remark}

For the proof of Theorem 3.1 we need a number of results which are also of interest themselves, in such a way that all together describe the different scenarios in the evolution of the equilibrium measure when $t$ grows from $0$ to $T = 1+\gamma$. Indeed, our main result, Theorem 3.1, is a synthesis of such a full description.

First, in \cite{MOR2015} and \cite{OrSL2015}, it was shown that the knowledge about the set of minima of the external field plays a key role in describing the evolution of the equilibrium measure $\lambda_t$ when $t$ varies. In those cases, the existence of two relative minima of the external field was shown as a sufficient condition for the existence of a two-cut phase (that is, a range of values of $t$ for which the support comprises two disjoint intervals). It will be also true in the current case (see Theorem \ref{thm:dynamics} below; in fact, the existence of a two-cut phase when the external fields has at least two minima is true for a more general setting).

First, taking into account that the relative minima of the external field $\varphi$ are roots of the polynomial
\begin{equation}\label{dphi}
P(x) = (x^2-1)( (x-1) + \gamma (x+1)) +  \gamma \beta_1^2 (x-1) + \beta_2^2 (x+1)\,,
\end{equation}
it is easy to see that the real relative minima lie on the interval $(-1,1)$. Moreover, we have,
\begin{theorem}\label{thm:2M}
Consider the external field \eqref{couple}, with $z_1=-1+\beta_1 i\,,\;z_1=1+\beta_2 i\,$ and $\gamma >0\,.$ Then,
\begin{itemize}
\item If $f(\beta_1^2,\beta_2^2) < 0\,,$ with $f$ given by \eqref{boundary}, or, what is the same, $(\beta_1,\beta_2)\in \Omega_0$, there exists two values $0<\widetilde{\Gamma}_1=\widetilde{\Gamma}_1(\beta_1,\beta_2)<\widetilde{\Gamma}_2 = \widetilde{\Gamma}_2(\beta_1,\beta_2)\,,$ such that \eqref{couple} has two real minima for $\gamma \in (\widetilde{\Gamma}_1,\widetilde{\Gamma}_2)$.
\item If $f(\beta_1^2,\beta_2^2) \geq 0\,,$ that is, $(\beta_1,\beta_2)\in \Omega_{\infty} \,\cup \,\mathcal{C}$, then \eqref{couple} has a single real minimum for any $\gamma > 0$.
\end{itemize}
\end{theorem}

\begin{remark}\label{rem:recipemin}
In a similar fashion as in the previous Theorem 3.1, a simple recipe to compute the critical values $\widetilde{\Gamma}_1$ and $\widetilde{\Gamma}_2$ is feasible. In fact, it is enough to compute suitable values of $\gamma$ in order to make the external field having a double critical point.
\end{remark}

Through the fact that the external field has two real minima is a sufficient condition for the existence of a two-cut phase, it is not a necessary one. Indeed, roughly speaking, if the external field is ``sufficiently non-convex'' is still feasible a two-cut phase. In order to get it, it is necessary the appearance of a new local extrema of \eqref{equilibrium}, that is, a double real root of polynomial $B$ outside the support, as it was said in Section 2. The border external fields in this sense are those for which a type III singularity (as defined above), that is, the confluence of a couple of imaginary zeros of $B$ with a zero of $A$ in the $AB$-representation given by \eqref{Cauchytr}, takes place for some critical value of $t$. Indeed, we have

\begin{theorem}\label{thm:typeIII}
For the equilibrium measure $\lambda_t$ in the external field \eqref{couple}, it holds:

\begin{itemize}
\item If $(\beta_1,\beta_2) \in \Omega_0\,,$ there exists two values $\Gamma_1, \Gamma_2\,,$ with $0<\Gamma_1=\Gamma_1(\beta_1,\beta_2)<\widetilde{\Gamma}_1<\widetilde{\Gamma}_2<\Gamma_2 = \Gamma_2(\beta_1,\beta_2)\,,$ such that for $\gamma = \Gamma_i\,,\,i=1,2,$ a type III singularity occurs at certain critical values of $t \in (0,T)$.
\item If $(\beta_1,\beta_2)\in \Omega_{\infty} \,\cup \,\mathcal{C}$, no type III singularity takes place.
\end{itemize}

\end{theorem}

\begin{remark}\label{rem:signderiv}

In the situations discussed in Theorem \ref{thm:typeIII}, \eqref{couple} has a single minimum and for $t$ sufficiently small we have a one-cut phase where polynomial $B$ has a couple of imaginary roots $b_1 = b$ and $b_2 = \overline{b}$. The sign of $\displaystyle \frac{\partial \Im b}{\partial t}$ plays a central role in the description of the dynamics of the support. In this sense, it is possible that $\Im b$ is always increasing and, thus, $b$ and $\overline{b}$ are always going away from the real axis, or $\Im b$ could be initially increasing but becomes decreasing at a certain moment, and so on. In this sense, for each value of $t$ there is a critical curve such that if $b = b(t)$ belongs to this curve, then $\displaystyle \frac{\partial \Im b}{\partial t} = 0\,.$ In previous \cite{MOR2015} and \cite{OrSL2015}, this curve takes the form of a hyperbola and a circle, respectively. In the present case, its shape is much more involved. Indeed, we have by \eqref{dynamics1},
\begin{equation*}\label{signderiv}
\frac{\partial \Im b}{\partial t}\,<\,0\;\Leftrightarrow\;\Re \left(\frac{D(b)}{A(b)}\right)\,>\,0\;\Leftrightarrow\;\Re D(b) \Re A(b)\,+\,\Im D(b) \Im A(b)\,>\,0\,,
\end{equation*}
where $A(b) = (b-a_1)(b-a_2)$ and $D(b) = (b^2-z_1^2)(b^2-z_2^2)$. Thus, in this case the critical curve is given in terms of a bivariate polynomial of degree $6$ in $x = \Re b$ and $y = \Im b$.

Therefore, now the geometry of the problem is much more involved. Furthermore, it is easy to check that, in this situation, while the support plays the role of a ``repellent'' for the couple of conjugate roots $b$ and $\overline{b}$, the rest of the real line acts as an attractor. Otherwise, for a general rational external field, the support always repels the couple of imaginary roots, while each gap is split in an odd number of subintervals by the roots of $F$ and $B$ in such a way that the first one acts as an ``attractor'', the second one as a ``repellent'', and so on (of course, it is also necessary to take also into account the multiplicity of each zero).

\end{remark}

Now, we have all the ingredients for describing the evolution of the equilibrium measure and, especially, its support $S_t$ when $t$ grows from $0$ to $T$. Our main result, Theorem \ref{thm:main}, is a simplified version of the following result.

\begin{theorem}\label{thm:dynamics}
Let $\beta_1,\beta_2 > 0$, and $0<\Gamma_1<\widetilde{\Gamma}_1<\widetilde{\Gamma}_2<\Gamma_2\,,$ as given in Theorems \ref{thm:2M}-\ref{thm:typeIII}. Then,
\begin{itemize}
\item [(a)] If $(\beta_1,\beta_2) \in \Omega_0\,$ and $\gamma \in (\widetilde{\Gamma}_1, \widetilde{\Gamma}_2)\,,$ we have the following phase diagram for the support of the equilibrium measure, $S_t$:

\textbf{one-cut}, for $\,t\in(0,T_1)\;\longrightarrow\;$ \textbf{two-cut}, for $\,t\in(T_1,T_2)\;\longrightarrow\;$
\textbf{one-cut}, for $\,t\in(T_2,T)\,.$

At $t=T_1$ ($t=T_2$), a type I (respect., type II) singularity occurs. If the external field $\varphi$ takes the same value in its two relative minima, then $T_1 = 0$ in the phase diagram below and the initial one-cut phase is absent.

\item [(b)] If $(\beta_1,\beta_2) \in \Omega_0\,$ and $\gamma \in (\Gamma_1, \widetilde{\Gamma}_1] \cup [\widetilde{\Gamma}_2, \Gamma_2)\,,$ the phase diagram for $S_t$ is the same as in (a), but the appearance of a pair of new local extrema occurs at a certain $t=T_0$, with $T_0 < T_1$.

\item [(c)] If $(\beta_1,\beta_2) \in \Omega_0\,$ and $\gamma \in (0,\,\Gamma_1] \cup [\Gamma_2,\,\infty)\,,$ \textbf{one-cut} phase holds for any $t\in (0,T)\,$.

\item [(d)] If $(\beta_1,\beta_2) \in \Omega_{\infty} \,\cup \,\mathcal{C}\,,$ we have \textbf{one-cut} phase for any $\gamma, t \in (0,T)\,.$

\end{itemize}

\end{theorem}

\begin{remark}\label{rem:totsym}
The so-called totally symmetric case studied in \cite{OrGa2010}, that is, when $z_2 = -\overline{z}_1$ and $\gamma = 1$ (equal heights and masses), may be now revisited in the light of results in Theorem \ref{thm:dynamics}. In this case, the symmetry of the external field is inherited by the support, what means that  $a_1 = -a_2 = -a\,$ and $b = 0\,$ when type II transition (fusion of the two cuts) occurs. Hence, \eqref{residues} yields
\begin{equation*}\label{systemTtotsym}
(2 -T_2)^2\,z_1^4\,(z_1^2-a^2)\,+\,16\,(\Im z_1)^2\,(\Re z_1)^2\,z_1\, =\,0\,,
\end{equation*}
and thus, the following system of nonlinear equations arises, with $T_2$ and $a$ as unknowns (of course, we are looking for solutions for which $T_2 < T = 2$),
\begin{equation}\label{totsymsyst}
\begin{cases} K\,\left((1-\beta^2)(1-\beta^2-a^2)-4\beta^2\right)\,+\,16\beta^2 = 0\,, \\ 2(1-\beta^2)-a^2 = 0\,, \end{cases}
\end{equation}
where $K = (2-T_2)^2 >0\,.$ From the second identity in \eqref{totsymsyst}, it is clear that necessarily $\beta < 1$. Under this condition, it is easy to check that the fusion of cuts takes place for
$$T_2 = 2\,\frac{(1-\beta)^2}{1+\beta^2}\,<\,2\,.$$ Finally, for $t\in (T_2,2)$, the one--cut phase takes place and Theorem 2.2 implies, for the density $\lambda'_t$ of the equilibrium measure, that
$$\lim_{t\rightarrow 2}\,\lambda'_t(x) = \,\frac{1}{\pi}\,\frac{x^2+\beta^2+1}{D(x)^2}\,,$$
with $D(x) = ((x+1)^2+\beta^2)\,((x-1)^2+\beta^2)$.

Thus, in the totally symmetric case, when $\beta < 1$, we always have the phase diagram:

\begin{center}
\textbf{two-cut} ($0<t<T_2$) $\;\longrightarrow\;$ \textbf{one-cut} ($T_2\leq t<T=2$)
\end{center}

On the other hand, when $\beta \geq 1$, it is easy to check that $\varphi$ only has a real critical point, at $x=0$, where it attains its absolute minimum. Therefore, the support $S_t$ starts being of the form $S_t = [-a(t),a(t)]\,,$ with $a(t)$ an increasing function as above. No phase transition occurs, since it would imply by symmetry a three-cut situation, which is not possible.

The reader can check that the conclusions above agree with the results in \cite{OrGa2010}.
\end{remark}

\section{Proofs}

Throughout this section the proofs of Theorems 2.2 and 3.2--3.4 above will be displayed. As it was said, they all together render the proof of the main result Theorem 3.1 and enrich it with auxiliary results which are of interest themselves.

\subsection{Proof of Theorem 2.2}

 As it was shown in \cite[Theorem 2 (3)]{Buyarov/Rakhmanov:99}, we have that $\mu_t$ is increasing and continuous in the weak topology of the set of measures with compact support in $\R$. In addition, for any Borel set $I\subset \R$, it holds $\mu_t(I)\leq \mu_t(\R) = t \in (0,T)\,$ and, thus, there exists $\displaystyle \lim_{t\rightarrow T}\mu_{t}$ in the sense mentioned above.

Now, let us show that $\displaystyle \lim_{t\rightarrow T} S_{t}\,=\,\R\,.$ For this, consider the function $$\phi_t(x) = V^{\mu_t}(x) + \varphi(x) - c_t\,,\;t\in (0,T)\,,\;x\in \R\,,$$ where $c_t$ is the extremal constant given by \eqref{equilibrium}. It is clear that $\phi_t(x)\geq 0\,,x\in \R$ and $t\in (0,T)$. Let $\tau \in (0,T)$ fixed. If $x\in S_{\tau}$, then $\phi_{\tau}(x) = 0$ and since the family of supports $\{S_t\}$ is increasing (\cite[Theorem 2, (1)]{Buyarov/Rakhmanov:99}), we have that $\phi_{t}(x) = 0, t \geq \tau\,.$ On the other hand, if $x\in \R \setminus S_{\tau}$, \cite[(1.8) and (1.13)]{Buyarov/Rakhmanov:99} yield,
$$\frac{\partial \phi_t(x)}{\partial t}|_{t=\tau^-}\,=\,-g_{\tau}(x)\,<\,0\,,\;\;\frac{\partial \phi_t(x)}{\partial t}|_{t=\tau^+}\,=\,-g^{\tau}(x)\,<\,0\,,$$
where $g_{\tau}$ ($g^{\tau}$) denotes de Green function of $\R \setminus S_{\tau}$ (respect., $\R \setminus S^{\tau}$) with pole at infinity, and
$\displaystyle S^{\tau} = \{x\in \R : \phi_{\tau}(x)= 0\}\,\supseteq \,S_{\tau}$. Hence, we have that $\phi_t(x)\geq 0\,,$ for any $x\in \R$ and $t\in (0,+\infty)$, and that $\phi_t(x)$ is a decreasing function of $t$ for any fixed $x\in \R$. This shows that there exists $$\lim_{t\rightarrow T}\,\phi_t(x) = \phi_T(x)\,,\,x\in \R\,.$$ Now, let us see that $\phi_T\,\equiv 0\,.$ To do it, recall that for $x\in S_t$, we have that $ V^{\mu_t}(x) + \varphi(x) = c_t$ and, thus,
\begin{equation}\label{decomposition}
V^{\lambda_t}(x) = -\,\frac{t}{T}\,\varphi(x)-\,\frac{T-t}{T}\,\varphi(x) + c_t = V^{\nu_t}(x)-\,\frac{T-t}{T}\,\varphi(x) + c_t\,,
\end{equation}
where $\displaystyle \nu_t=\sum_{j=1}^{N}\frac{t\gamma_{j}}{T} \delta_{z_j}\,,\;\;\text{with}\;\; \nu_t(\mathbb{C})=t$. Since $\varphi$ has an absolute minimum on the real axis, $\displaystyle m = \min_{x\in \R}\,\varphi (x) > \sum_{j=1}^q\,\gamma_j\,\Im z_j\,>\,-\infty\,,$ then \eqref{decomposition} implies that $$V^{\lambda_t}(x)\leq V^{\nu_t}(x)-\,\frac{T-t}{T}\,m + c_t\,,\,x\in S_t\,.$$
Thus, the Domination Principle \cite[Theorem II.3.2]{Saff:97} asserts that this inequality holds for any real $x$ (and, in fact, for any complex $x$).
Thus, $$\phi_t(x)\leq \,\frac{T-t}{T}\,(\varphi (x) - m)\,,\,x\in \R\,,$$ which shows that $$\lim_{t\rightarrow T}\,\phi_t(x) = 0\,,\,x\in \R\,$$ and, then, that $\displaystyle S^T = \R$.

Now, we are going to prove the rest of the results. First, we are dealing with the limit function $\displaystyle \widehat{\mu}_T = \lim_{t\rightarrow T} \widehat{\mu}_t\,.$
Taking into account \eqref{Cauchytr} and the previous analysis, we have that $\displaystyle \widehat{\mu}_T$ must be analytic on $\displaystyle\mathbb{C}\setminus(\mathbb{R}\cup\bigcup_{j=1}^N\{z_j,\overline{z}_j\})\,,$ in such a way that

\begin{itemize}
\item $\widehat{\mu}_T (x^+)=-\widehat{\mu}_T (x^-)\;$ for   $\;x\in\mathbb{R}$,
\item $\widehat{\mu}_T (x^+)\in\mathbb{R}^+i\,,\;$ taking into account the positivity of the measure,
\item For $z\to z_j\,,$ $$\widehat{\mu}_T (z)=\frac{\gamma_j}{2(z-z_j)}+O(1)$$
\item For $z\to \overline{z}_j\,,$
$$\widehat{\mu}_T (z)=\frac{\gamma_j}{2(z-\overline{z}_j)}+O(1)\,.$$
\end{itemize}

Therefore, having in mind the Liouville Theorem and some immediate consequences, we have that

$$\widehat{\mu}_T (z) = \begin{cases}  \sum_{j=1}^N\left(\frac{\gamma_j}{2(z-z_j)}-
\frac{\gamma_j}{2(z-\overline{z}_j)}\right)\,,\Im z>0\,,\\
-\sum_{j=1}^N\left(\frac{\gamma_j}{2(z-z_j)}-
\frac{\gamma_j}{2(z-\overline{z}_j)}\right)\,, \Im z<0\,. \end{cases}$$

Now, the conclusions easily follow.

\subsection{Proof of Theorem 3.2}

Polynomial $P$ in \eqref{dphi} may be rewritten in the form
\begin{equation}\label{decomposition2}
P(x) = P(x,\gamma) = (x+1)\left((x-1)^2+\beta_2^2\right)\,+\,\gamma\,(x-1)\left((x+1)^2+\beta_1^2\right)\,=\,u(x)\,+\,\gamma\,v(x)\,.
\end{equation}
Let us study the zeros of \eqref{decomposition2} when $\gamma$ increases. For $\gamma = 0$, $P$ has a single real zero at $x=-1$ and a couple of conjugate imaginary zeros at $z_2$ and $\overline{z}_2$, while when $\gamma$ tends to infinity, the real zero approaches $x=1$ and the couple of imaginary zeros tend to $z_1$ and $\overline{z}_1$. Let us start with $\gamma > 0$ small enough. Since
\begin{equation}\label{decreasing}
\displaystyle \frac{\partial P}{\partial \gamma}\, = \, v(x) \, <\,0 \,,\;\;\text{for}\;\;x<1\,,
\end{equation}
the real zero, say $\zeta_1$, move to the right as $\gamma$ increases. On the other hand, writing \eqref{decomposition2} in powers of $x$ yields that arithmetic mean of the zeros of $P$ equals $\displaystyle \frac{1}{3}\,\frac{1-\gamma}{1+\gamma}$  and, thus, this mean decreases as $\gamma$ increases, which means that the real parts of the couple of imaginary roots $\xi,\overline{\xi}$, move to the left. It implies there are just two possible scenarios for the evolution of the critical points of $\varphi$ as $\gamma$ travels across $(0,+\infty)\,.$
\begin{itemize}
\item The pair of imaginary roots never reach the real axis. In such a case, $\varphi$ has a single minimum for any $\gamma >0\,.$
\item There exist a real number $\widetilde{\Gamma}_1 > 0\,,$ such that for $\gamma = \widetilde{\Gamma}_1$ the pair of imaginary roots reach the real axis, giving birth to a double real root $\xi$ located to the right of the simple real root $\zeta_1$ (this is due to the fact that $P(x) < 0$ to the left of this simple real root with decreasing values of $P(x)$ as $\gamma$ increases). Immediately after the collision, a pair of new simple real roots arise, say $\zeta_2,\zeta_3$, in such a way that $-1<\zeta_1<\zeta_2<\xi<\zeta_3<1$ and with $\zeta_3$ moving to the right and $\zeta_2$ to the left (because of \eqref{decreasing}). This situation holds until $\zeta_1$ and $\zeta_2$ collide, creating a double real root for immediately going to $\C \setminus \R$.
\end{itemize}
It is easy to see that the boundary between these possible evolutions is the case where the pair of imaginary roots $\xi,\overline{\xi}$ collide with the real one, $\zeta_1$, giving birth to a triple real root for a certain value of $\gamma$: that is, when $(\beta_1,\beta_2)\in \mathcal{C}$, with $\mathcal{C}$ given by \eqref{curveC}. It is also easy to check that the region where two minima are feasible is $\Omega_0$ (consider, for instance, the case with $\beta_2 = \beta_1$ taking small positive values).

\subsection{Proof of Theorem 3.3}

This is, in fact, the most important theorem in order to prove our main result (Theorem 3.1). Its proof will be a consequence of Proposition \ref{prop:main} below, to whose proof it is devoted the most part of this Section. We know that when the external field has two minima, then a two-cut phase occurs, but this is not the unique way to reach that phase. Indeed, when the external field has a single minimum, the existence of such a phase is equivalent to the birth of a new minimum of the total potential \eqref{equilibrium} in a previous ``instant'' $t$. Proposition 4.1 below analyzes the possible birth of this minimum of \eqref{equilibrium} as $\gamma$ varies.

The setting for the result below is as follows.

Let $(\beta_1,\beta_2)\in (\R^+)^2$ be fixed, and suppose that for a certain $\gamma = \gamma_0$ and $t=t_0$ the polynomial $B$ in the $AB$--representation \eqref{Cauchytr} has a multiple real root (i.e., double or triple), not belonging to the interior of the support. Now, let $I = I(\beta_1,\beta_2)$ the largest interval containing $\gamma_0$ such that $B$ has a multiple real root (outside the interior of the support, too) for some $t=t(\gamma)$. In this setting, let us denote, as usual, by $a_1$ and $a_2$ the endpoints of the support (one--cut) and by $b$ the multiple root of $B$. Then, we have,
\begin{proposition}\label{prop:main}
The interval $I$ is compact and the functions $a_1=a_1(\gamma)$, $a_2=a_2(\gamma)$, $b=b(\gamma)$ (see \eqref{Cauchytr}) and $t=t(\gamma)$ are analytic in the interior of $I$ and continuous in $I$, with $a_1(t)$, $a_2(t)$ and $b(t)$ being monotonic. In particular, using now the notation $\displaystyle \dot{f} = \frac{\partial f}{\partial \gamma}\,,$ it holds,
 \begin{itemize}
 \item If $a_1\leq a_2\leq b\,,$ then $\dot{a}_1>0$, $\dot{a}_2<0$, $\dot{b}>0$ and there exists $\Gamma_1 >0$ such that $I=[\Gamma_1,\tilde{\Gamma}_1]$, in such a way that a type III singularity takes place for $\gamma = \Gamma_1$ and a certain value of $t$.
 \item If $b\leq a_1\leq a_2\,,$ then $\dot{a}_1<0$, $\dot{a}_2>0$, $\dot{b}>0$ and there exists $\Gamma_2 >0$ such that $I=[\tilde{\Gamma_2},\Gamma_2]$, in such a way that a type III singularity occurs for $\gamma = \Gamma_2$ and a certain value of $t$.
 \end{itemize}

\end{proposition}

The following result, which in turn yields Theorem 3.3, is a direct consequence of Proposition 4.1 and Theorem 3.2.

\begin{corollary}\label{cor:equivalence}
Let $(\beta_1,\beta_2)\in (\R^+)^2$ be fixed. Then, the following statements are equivalent:
\begin{itemize}
\item[i)] $(\beta_1,\beta_2)\in \Omega_0$
\item[ii)] There exists $\gamma>0$ such that $\phi'$ has a double root.
\item[iii)] There exist $\gamma>0$ and $t\geq 0$ for which polynomial $B$ in \eqref{Cauchytr} has a multiple real root.
\item[iv)] There exist $\gamma>0$ such that a type III singularity takes place for some $t>0$.
\end{itemize}
Furthermore, if some of these statements holds, there exist exactly two values of $\gamma$ satisfying it.
\end{corollary}
Therefore, the two-cut phase is feasible when $(\beta_1,\beta_2)\in \Omega_0\,,$ and the number of type III singularities is $2$ at most.

Now, let us proceed with the proof of Proposition 4.1. Since it deals with the case where polynomial $B$ in \eqref{Cauchytr} has a double (at least) root $b$, let us start pointing out that in this case \eqref{residues} implies that the following system of equations holds
\begin{equation}\label{typeIV}
\begin{cases} (T-t)\,(z_1-b)^2\,\sqrt{(z_1-a_1)(z_1-a_2)}\,-\,i\, (z_1-z_2)(z_1-\overline{z_2})\, \Im z_1 & = 0\,,\\[.3cm]
(T-t)\,(z_2-b)^2\,\sqrt{(z_2-a1)(z_2-a_2)}\,-\,i\,\gamma \, (z_2-z_1)(z_2-\overline{z_1}) \,\Im z_2 & = 0\,. \end{cases}
\end{equation}
First, we need the following technical results. On the sequel, $\Arg z$ denotes the branch of the argument of the complex number $z$ belonging to $(-\pi,\pi]$.
\begin{lemma}\label{lem:cotsuma}
Let $c,d \in \R$. Then,
\begin{align}
&\Arg(z_1-c)+\Arg(z_1-d)<\pi \;\;\text{if and only if}\;\; c+d<-2\,,\label{cotsuma1}\\
&\Arg(z_2-c)+\Arg(z_2-d)<\pi \;\;\text{if and only if}\;\; c+d<2\,,\label{cotsuma2}
\end{align}
\end{lemma}
\begin{proof}
We know that
$$\Arg(z_j-c)+\Arg(z_j-d)\in(0,2\pi)\,.$$ Thus, making use of well-known trigonometric identities, we have,
 \begin{align*}
    &\sin\left(\Arg(z_j-c)+\Arg(z_j-d)\right)\\
    =&
    \frac{\beta_j(2\Re z_j-c-d)}{\sqrt{(\Re z_j-c)^2+\beta_j^2}\sqrt{(\Re z_j-d)^2+\beta_j^2}}>0\,,
    \end{align*}
which shows that
$$\Arg(z_j-c)+\Arg(z_j-d)<\pi \;\;\text{iff}\;\;  2\Re z_j-c-d>0 \;\;\text{iff}\;\; c+d<2\Re z_j\,,$$
and it settles the proof.
\end{proof}
Now, it is convenient to introduce the point $$x_0 =\frac{-\beta_1^2+\beta_2^2}{4}\,,$$ that is, the intersection between the mediatrix of the segment joining $[z_1,z_2]$ and the real axis, and the points $x_1 < x_2$, where the circumference with center at $x_0$ and passing through $z_1$ and $z_2$ meets the real axis (see Figure 2 below). It is also worth to point out that
$$\Arg(z_j-x_1)\,=\,\frac{1}{2}\,\Arg(z_j-x_0)\,,\;\Arg(z_j-x_2)\,=\,\frac{\pi}{2}\,+\,\frac{1}{2}\,\Arg(z_j-x_0)\,,\;j=1,2\,.$$
\begin{lemma}\label{lem:argum}
Suppose that for some fixed $(\beta_1,\beta_2,\gamma)\in (\R^+)^3$ polynomial $B$ has a double root $b$. Then, it holds
\begin{align}
 &\frac{1}{2}\Arg(z_1-a_1)+\frac{1}{2}\Arg(z_1-a_2)+2\Arg(z_1-b)-\Arg(z_1-x_0)-\frac{3\pi}{2}=0\,,\label{eqsargs1}\\
 &\frac{1}{2}\Arg(z_2-a_1)+\frac{1}{2}\Arg(z_2-a_2)+2\Arg(z_2-b)-\Arg(z_2-x_0)-\frac{\pi}{2}=0\,,\label{eqsargs2}
 \end{align}
 where $a_1,a_2$ denote the endpoints of the support $S_t$.
\end{lemma}
\begin{proof}
From \eqref{typeIV}, the following system must hold:
\begin{align}
 &\frac{1}{2}\Arg(z_1-a_1)+\frac{1}{2}\Arg(z_1-a_2)+2\Arg(z_1-b)-\Arg(z_1-x_0)-\frac{3\pi}{2}=2k_1\pi\,,\label{systarg1}\\
 &\frac{1}{2}\Arg(z_2-a_1)+\frac{1}{2}\Arg(z_2-a_2)+2\Arg(z_2-b)-\Arg(z_2-x_0)-\frac{\pi}{2}=2k_2\pi\,,\label{systarg2}
 \end{align}
 with $k_j\in\mathbb{Z}$. Now, it will be shown that $k_1=k_2=0$. First, let us see that these $k_j$ just can take some particular values.

 We initially deal with the \eqref{systarg1}. First, since
 $\Arg(z_1-a_1)\in(0,\pi)$, $\Arg(z_1-a_2)\in(\pi/2,\pi)$ (because $a_2(0)>-1$ and $\partial a_2/\partial t>0$), $\Arg(z_1-b)\in(0,\pi)$ and $\Arg(z_1-x_0)\in (0,\pi)$, it yields
        \begin{align*}
        &\frac{1}{2}\Arg(z_1-a_1)+\frac{1}{2}\Arg(z_1-a_2)+2\Arg(z_1-b)
        -\Arg(z_1-x_0)-\frac{3\pi}{2}\\
        \in&
        \left(\frac{-9\pi}{4},\frac{3\pi}{2}\right)\,,
        \end{align*}
 and we conclude that $k_1=-1$ or $k_1=0$. In a similar way, it is easy to check that $k_2 \in \{0,1\}$ in \eqref{systarg2}.

 Now, let us show that $k_1=-1$ cannot occur. Let us see, first, that if $k_1=-1$, then we would necessarily have that
 $a_1+a_2<-2$, $b<x_1$. Indeed,

 \begin{itemize}

  \item We have $a_1+a_2<-2$, since otherwise, \eqref{cotsuma1} would yield
        $$\frac{1}{2}\Arg(z_1-a_1)+\frac{1}{2}\Arg(z_1-a_2)\geq \frac{\pi}{2}\,,$$
  and thus,
      \begin{align*}
      &\frac{1}{2}\Arg(z_1-a_1)+\frac{1}{2}\Arg(z_1-a_2)+2\Arg(z_1-b)
    -\Arg(z_1-x_0)-\frac{3\pi}{2}\\
    \geq& \frac{\pi}{2}+2\Arg(z_1-b) -\Arg(z_1-x_0)-\frac{3\pi}{2}=2\Arg(z_1-b) -\Arg(z_1-x_0)-\pi\\
    >&0-\pi-\pi=-2\pi\,,
    \end{align*}
  which would imply that $k_1\neq -1$.

  \item $b<x_1\,,$ since
  \begin{align*}
    &-2\pi=\frac{1}{2}\Arg(z_1-a_1)+\frac{1}{2}\Arg(z_1-a_2)+2\Arg(z_1-b)
    -\Arg(z_1-x_0)-\frac{3\pi}{2}\\
    \Longrightarrow&2\Arg(z_1-b)=-\frac{1}{2}\Arg(z_1-a_1)-\frac{1}{2}\Arg(z_1-a_2)
    +\Arg(z_1-x_0)-\frac{\pi}{2}\\
    \Longrightarrow& \Arg(z_1-b)<\frac{1}{2}\Arg(z_1-x_0)=\Arg(z_1-x_1)\Longrightarrow b<x_1.
  \end{align*}

 \end{itemize}

 Thus, taking into account \eqref{systarg2}, one has,
 \begin{align*}
 2k_2\pi&=\frac{1}{2}\Arg(z_2-a_1)+\frac{1}{2}\Arg(z_2-a_2)+2\Arg(z_2-b)
        -\Arg(z_2-x_0)-\frac{\pi}{2}\\
        &<\frac{1}{2}\Arg(z_2-a_1)+\frac{1}{2}\Arg(z_2-a_2)+2\Arg(z_2-x_1)
        -\Arg(z_2-x_0)-\frac{\pi}{2}\\
        &=\frac{1}{2}\Arg(z_2-a_1)+\frac{1}{2}\Arg(z_2-a_2)-\frac{\pi}{2}<0\,,
 \end{align*}
where for the last inequality we have used \eqref{cotsuma2}.
But this inequality would imply that $k_2<0$, while it is known that $k_2=0$ or $k_2=1$. Hence, we conclude that $k_1=0$ and \eqref{eqsargs1} is established.

In a similar fashion, \eqref{eqsargs2} is established.

\end{proof}

It will be also useful the following result about the location of the point $b$ and the arithmetic mean of the endpoints of $S_t$.
\begin{lemma}\label{lem:position}

\hspace{.2cm}

\begin{itemize}
\item[i)] The point $\displaystyle \frac{a_1+a_2}{2}\in (-1,1)$ or, equivalently,
\begin{align*}
 &\Arg(z_1-a_1)+\Arg(z_1-a_2) > \pi\,,\\
 &\Arg(z_2-a_1)+\Arg(z_2-a_2) < \pi\,.
 \end{align*}
\item[ii)] $b\in (-1,1)$.
\end{itemize}
\end{lemma}
\begin{proof}
Let us show, first, that $b\in(x_1,x_2)$. Indeed, \eqref{eqsargs1} yields
\begin{align*}
2\Arg(z_1-b)&=
-\frac{1}{2}\Arg(z_1-a_1)-\frac{1}{2}\Arg(z_1-a_2)+\Arg(z_1-x_0)+\frac{3\pi}{2}\\
&>-\frac{\pi}{2}-\frac{\pi}{2}+\Arg(z_1-x_0)+\frac{3\pi}{2}=\Arg(z_1-x_0)+\frac{\pi}{2}>\Arg(z_1-x_0)\,,
\end{align*}
which implies that
$$\Arg(z_1-b)>\frac{1}{2}\Arg(z_1-x_0)=\Arg(z_1-x_1)$$
and, hence, it holds $b>x_1$. Analogously, from \eqref{eqsargs2} it is easy to get that $b<x_2$.

Although this first bound for $b$ is rough, it allows to get the bounds for the arithmetic mean. In fact, from \eqref{eqsargs1} and using the bounds obtained for $b$, one has,
\begin{align*}
\frac{1}{2}\Arg(z_1-a_1)+\frac{1}{2}\Arg(z_1-a_2)&=
-2\Arg(z_1-b)+\Arg(z_1-x_0)+\frac{3\pi}{2}\\
&>-2\Arg(z_1-x_2)+\Arg(z_1-x_0)+\frac{3\pi}{2}\\
&=
-2\left(\frac{\pi}{2}+\frac{1}{2}\Arg(z_1-x_0)\right)+\Arg(z_1-x_0)+\frac{3\pi}{2}=\frac{\pi}{2}\,,
\end{align*}
which, by \eqref{cotsuma1}, yields $(a_1+a_2)/2>-1$. In the same way, from \eqref{eqsargs2} and \eqref{cotsuma2}, the inequality $(a_1+a_2)/2<1$ is easily obtained.

Finally, using these bounds for the mass center of the endpoints, it is possible to precise the location of $b$. Let us start showing that $b<1$. Indeed, if $b\geq 1$, we have, using \eqref{eqsargs2},
\begin{align*}
0&\geq \frac{1}{2}\Arg(z_2-a_1)+\frac{1}{2}\Arg(z_2-a_2)+2\frac{\pi}{2}-\Arg(z_2-x_0)-\frac{\pi}{2}\\
&=
\frac{1}{2}\Arg(z_2-a_1)+\frac{1}{2}\Arg(z_2-a_2)-\Arg(z_2-x_0)+\frac{\pi}{2}
\end{align*}
and, hence,
\begin{equation}\label{medas}
\frac{1}{2}\Arg(z_2-a_1)+\frac{1}{2}\Arg(z_2-a_2)\leq\Arg(z_2-x_0)-\frac{\pi}{2}\,.
\end{equation}
Inequality \eqref{medas} implies some consequences. First, it is easy to check that it would be possible as long as $\Arg(z_2-x_0)>\pi/2$, which means that $x_0>1$.

Moreover, let $a\in\R$ such that
$$\frac{1}{2}\Arg(z_2-a_1)+\frac{1}{2}\Arg(z_2-a_2)=\Arg(z_2-a)\,,$$
that is, $a \in (a_1,a_2)$ is the point where the bisector of the angle $\widehat{a_1,z_2,a_2}$ meets the real axis. On the other hand, let $\tilde{a}\in\mathbb{R}$ such that
$$\Arg(z_2-x_0)-\frac{\pi}{2}=\Arg(z_2-\tilde{a})\,.$$ This last point may be seen as the point where the tangent line to the circumference with center $x_0$ and passing through $z_2$ intersects the real axis. Since $x_0>1$, then $\tilde{a} < x_1$ (see Figure \ref{figas}). Thus, inequality \eqref{medas} yields $a\leq\tilde{a}<x_1<1$, but this is a contradiction with the fact that $a_1+a_2>1$. Indeed, it is enough to make use of the following property from elementary geometry:

``Let $ABC$ be a triangle and consider the bisector of the angle $a$, which splits the segment $BC$ into two parts, $BD$ and $DC$. Then, the length of $BD$ is less than the length of $DC$ if and only if the angle $B$ is greater than the angle $C$''.

This simple property applied to the triangle $a_1a_2z_2$, and taking the bisector joining $z_2$ with $a$, means that $a_2-a<a-a_1$ but, then,
$$\frac{a_1+a_2}{2}=a-\frac{(a-a_1)-(a_2-a)}{2}<a\leq\tilde{a}<x_1<-1,$$ which contradicts the result in Lemma 4.1 above.

Proceeding in an analogous way, we can prove the lower bound, i.e. $b>-1$, using now \eqref{eqsargs1}.

\end{proof}
\begin{figure}\label{fig:location}
        \begin{center}
        \includegraphics[scale=0.4]{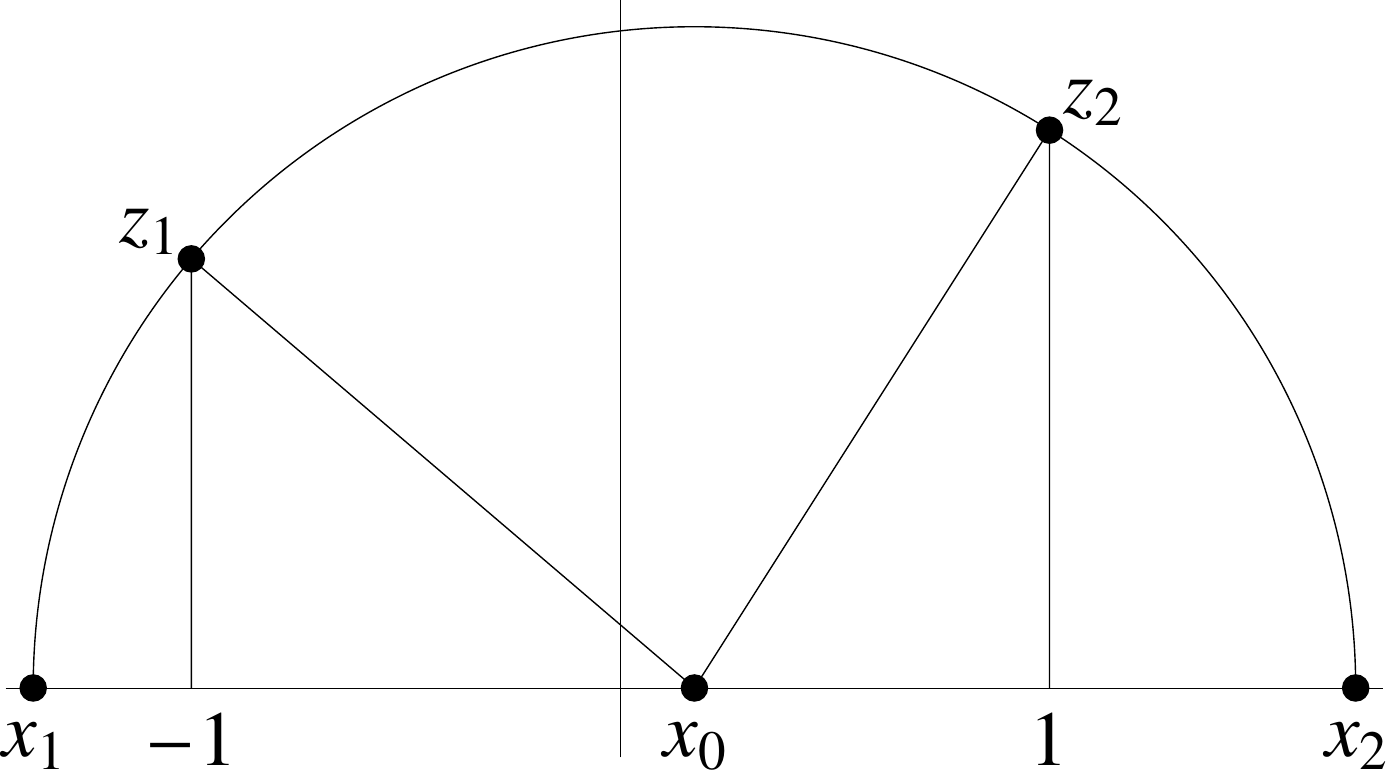}
        \caption{Location of points $x_0,x_1,x_2$.}
        \end{center}
\end{figure}

\begin{figure}[h]
\begin{center}
\includegraphics[scale=0.7]{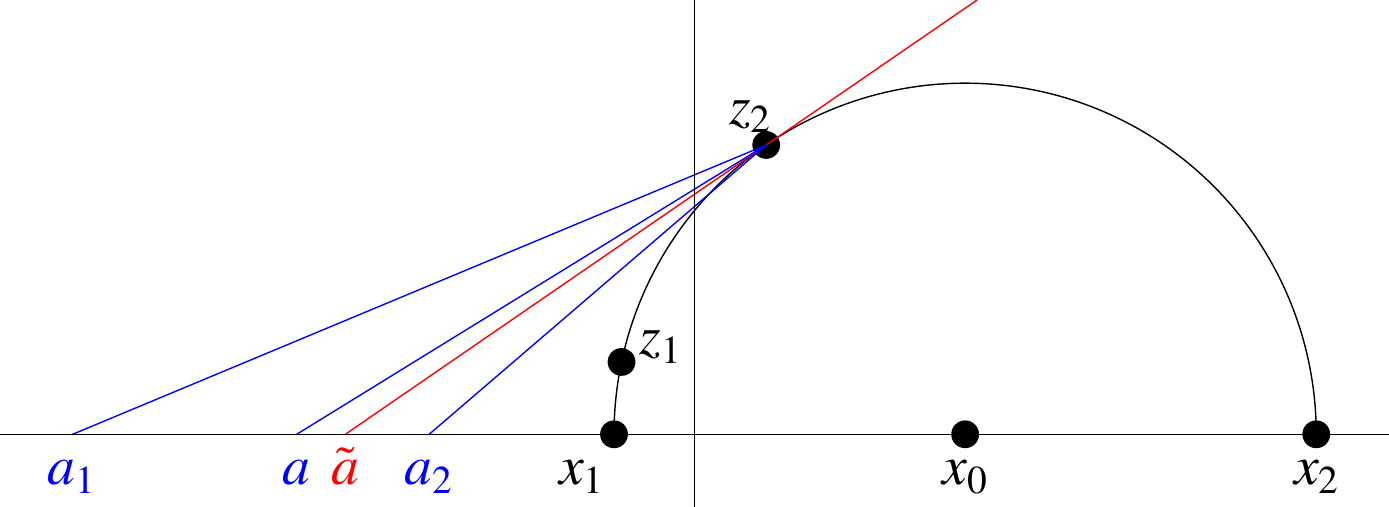}
\caption{location of points $a$ and $\widetilde{a}$ if $b\geq 1$.}
\label{figas}
\end{center}
\end{figure}

\subsubsection*{Proof of Proposition 4.1}

Let us take $(\beta_1,\beta_2)$ fixed and $\gamma$ and $t$ varying in such a way that $B$ has a double root: $B(z)=(z-b)^2\,,$ with $b\notin (a_1,a_2)$.

First, from \eqref{parammass} we have
$$\frac{\partial}{\partial \gamma}\frac{(T-t)\sqrt{A(z)}(z-b)^2}{D(z)}= \,\frac{H(z)}{\sqrt{A(z)}D_2(z)}\,$$
with $D_2=(z-z_2)(z-\overline{z_2})$ and $H$ being a monic polynomial of degree $2$ and having at least a root in $(a_1,a_2)$. Taking residues at $z_1$ and $z_2$ in the previous identity it is easy to check that $H(z)=D_2(z)+P(z),$ where $P$ is the interpolating polynomial of degree $\leq 1$ satisfying
$$P(z_2)=\frac{\sqrt{A(z_2)}(z_2-\overline{z_2})}{2}\,,\qquad P(\overline{z_2})=\frac{\sqrt{A(\overline{z_2})}(\overline{z_2}-z_2)}{2}$$
and, thus, $$H(z)=(z-z_2)(z-\overline{z_2})+\frac{\sqrt{A(z_2)}}{2}(z-\overline{z_2})
 +\frac{\sqrt{A(\overline{z_2})}}{2}(z-z_2)\,.$$
On the other hand, from \eqref{BuyRakhrat} we also have
$$\frac{\partial}{\partial t}\frac{(T-t)\sqrt{A(z)}(z-b)^2}{D(z)}=-\,\frac{1}{\sqrt{A(z)}}\,.$$
Therefore, considering now $t=t(\gamma)$ and using again the notation $\dot{f}=\partial f/\partial \gamma$, one has, on the one hand,

\begin{align*}\label{deriv1}
\frac{\partial}{\partial \gamma}\frac{(T-t)\sqrt{A(z)}(z-b)^2}{D(z)}=\frac{H(z)}{\sqrt{A(z)}D_2(z)}+\frac{-1}{\sqrt{A(z)}}\dot{t}=
\frac{H(z)-\dot{t} D_2(z)}{\sqrt{A(z)}D_2(z)}=:\frac{L(z)}{\sqrt{A(z)}D_2(z)}
\end{align*}
and, on the other,
\begin{equation*}
\begin{split}
& \frac{\partial}{\partial \gamma}\frac{(T-t)\sqrt{A(z)}(z-b)^2}{D(z)}= \\
& \frac{1}{D(z)}\left(-\dot{t}\sqrt{A(z)}(z-b)^2+(T-t)\frac{\dot{A}(z)}{2\sqrt{A(z)}}(z-b)^2+(T-t)\sqrt{A(z)}2(z-b)(-\dot{b})\right)\,.
\end{split}
\end{equation*}
Then, it yields,
\begin{equation}\label{EDOpolgamma}
-2\dot{t}A(z)(z-b)^2+(T-t)\dot{A}(z)(z-b)^2-4(T-t)A(z)(z-b)\dot{b}=2L(z)D_1(z)
\end{equation}
 where $D_1(z)=(z-z_1)(z-\overline{z_1})=(D/D_2)(z)$. The left and right-hand sides of \eqref{EDOpolgamma} are polynomials of degree $4$, with the first-hand one vanishing for $z=b$. This implies that $L(b)=0$ and thus,
\begin{equation*}\label{EDOgammat}
\dot{t}=\frac{H(b)}{D_2(b)}=\frac{D_2(b)+P(b)}{D_2(b)}\,.
\end{equation*}
Therefore, the following differential equations holds
\begin{equation*}\label{EDOgamma}
\begin{array}{ll}
\dot{a}_1&=\displaystyle\frac{-2D_1(a_1)L(a_1)}{(T-t)(a_1-a_2)(a_1-b)^2}\,,\\[5mm]
\dot{a}_2&=\displaystyle\frac{-2D_1(a_2)L(a_2)}{(T-t)(a_2-a_1)(a_2-b)^2}\,,\\[5mm]
\dot{b}&=\displaystyle\frac{-D_1(b)L'(b)}{2(T-t)(b-a_1)(b-a_2)}\,.
\end{array}
\end{equation*}
Now, since polynomial $L$ is very important for our analysis, we are concerned with its expression. Indeed, we have for $x\in \R$,
\begin{align*}
L(x)&=H(x)-\dot{t}D_2(x)=P(x)-\frac{P(b)}{D_2(b)}D_2(x)\\
&=
\Re\left(\sqrt{A(z_2)}(x-\overline{z_2})\left(1-\frac{x-z_2}{b-z_2}\right)\right)\\
&=
(b-x)\Re\left(\sqrt{A(z_2)}\frac{x-\overline{z_2}}{b-z_2}\right)\,.
\end{align*}
Thus, if $\ell$ denotes the other real root of $L$, then
$$L(x)=\Re\left(\frac{\sqrt{A(z_2)}}{z_2-b}\right)(x-b)(x-\ell)\,,$$
and $\ell$ is such that
$$
\Re\left(\sqrt{A(z_2)}\frac{\ell-\overline{z_2}}{b-z_2}\right)= 0
\,.$$
Consequently, taking into account Lemma 4.3, the following relation for the arguments holds:
\begin{equation}\label{relarg}
\frac{1}{2}\,\Arg (z_2-a_1)\,+\,\frac{1}{2}\,\Arg (z_2-a_2)\,-\,\Arg (z_2-\ell)\,-\,\Arg (z_2-b)\,=\,-\,\frac{\pi}{2}\,.
\end{equation}
Now, we are in a position to study what happens in each of the scenarios in the statement of Proposition 4.1.
\begin{itemize}

\item If $a_1\leq a_2\leq b$ holds, then
$$\Arg\left(\frac{\sqrt{A(z_2)}}{z_2-b}\right)=
\frac{1}{2}\Arg(z_2-a_1)+\frac{1}{2}\Arg(z_2-a_2)-
\Arg(z_2-b)\in\left(-\,\frac{\pi}{2},0\right)\,,$$
and thus, the leading coefficient of $L$ is
given by $$\Re\frac{\sqrt{A(z_2)}}{z_2-b}>0\,.$$

Now, let us show that $l>b$. Indeed, if this were not the case, by \eqref{relarg} we would have,
$$\frac{1}{2}\Arg(z_2-a_1)+\frac{1}{2}\Arg(z_2-a_2)+\frac{\pi}{2}\leq 2\Arg(z_2-b)$$
but, using \eqref{eqsargs2}, $$2\Arg(z_2-b)\,=\,-\,\frac{1}{2}\Arg(z_2-a_1)\,\frac{1}{2}\Arg(z_2-a_2)\,+\,\Arg(z_2-x_0)\,+\,\frac{\pi}{2}$$ and, hence,
\begin{align*}
&\frac{1}{2}\Arg(z_2-a_1)+\frac{1}{2}\Arg(z_2-a_2)+\frac{\pi}{2}\leq \frac{-1}{2}\Arg(z_2-a_1)+\frac{-1}{2}\Arg(z_2-a_2)+\Arg(z_2-x_0)+\frac{\pi}{2}\\
\Longrightarrow &
\frac{1}{2}\Arg(z_2-a_1)+\frac{1}{2}\Arg(z_2-a_2)\leq \frac{1}{2}\Arg(z_2-x_0)=\Arg(z_2-x_1)
\end{align*}
Thus, taking $a$ such that $\displaystyle \frac{\Arg(z_2-a_1)+\Arg(z_2-a_2)}{2}=\Arg(z_2-a)$, we would get $a_2-a<a-a_1$ and, hence,
$$\frac{a_1+a_2}{2}=a+\frac{(a2-a)-(a-a_1)}{2}<a<x_1<1\,,$$
which is not possible. Therefore, the inequality $l>b$ has been established, and it yields.
\begin{equation*}
\begin{array}{ll}
\dot{a}_1&=\displaystyle\frac{-2D_1(a_1)L(a_1)}{(T-t)(a_1-a_2)(a_1-b)^2}>0\,,\\[5mm]
\dot{a}_2&=\displaystyle\frac{-2D_1(a_2)L(a_2)}{(T-t)(a_2-a_1)(a_2-b)^2}<0\,,\\[5mm]
\dot{b}&=\displaystyle\frac{-D_1(b)L'(b)}{2(T-t)(b-a_1)(b-a_2)}>0\,,
\end{array}
\end{equation*}
Moreover, having in mind \eqref{eqsargs2}, the fact that $\dot{b}>0$ necessarily implies that $\dot{a}<0$,
what means, by \eqref{relarg}, that $\Arg(z_2-\ell)$ is a decreasing function of $\gamma$ and, hence, $\dot{\ell}<0$.

Therefore, considering the setting where $(\beta_1,\beta_2)$ are fixed and $\gamma$ is such that there exists $t>0$ with $a_1<a_2<b$, if $\gamma$ is allowed to increase as far as possible, we see that the endpoints $a_1$ and $a_2$ tend to collide, as well as points $b$ and $\ell$ on the interval $(a_2,+\infty)$; but we know this last collision cannot take place since the inequality $b<\ell$ is strict. Therefore, the unique feasible final setting consists in the collision $a_1=a_2$, which obviously occurs when $t=0$: we reach the situation where $\varphi$ has two minima (for $\gamma = \widetilde{\Gamma}_1$ in Theorem 3.2).

On the other hand, if the same scenario is handled but now allowing $\gamma$ to decrease as far as possible, it easy to check that the endpoints $a_1$ and $a_2$ tend to move away from each other, as well as points $b$ and $\ell$, in such a way that necessarily the collision between $a_2$ and $b$ finally occurs and, so, a type III singularity takes place (for $\gamma = \Gamma_1 < \widetilde{\Gamma}_1$).

\item Finally, the reciprocal case $b\leq a_1 \leq a_2$ may be easily reduced to the previous one by means of the transformation $x\,\rightarrow\,-x$, with $\gamma \rightarrow \frac{1}{\gamma}$, which yields $$a_1 \rightarrow -a_2\,,\;a_2 \rightarrow -a_1\,,\;b \rightarrow -b\,,\;t \rightarrow \frac{t}{\gamma}\,.$$

\end{itemize}

\subsection{Proof of Theorem 3.4}

The full description of the dynamics of the equilibrium measure runs parallelling to the proof of \cite[Theorems 15--16]{MOR2015} and \cite[Theorem 2.1]{OrSL2015}. Therefore, we restrict here to outline the proof, omitting certain details.

When $(\beta_1,\beta_2) \in \Omega_0\,$ and $\gamma \in (\widetilde{\Gamma}_1, \widetilde{\Gamma}_2)\,,$ Theorem \ref{thm:2M} shows that \eqref{couple} has two relative minima $-1<\zeta_1<\zeta_3<1$. First, assume that $\varphi(\zeta_1) < \varphi(\zeta_3)\,$ and, thus, that the leftmost relative minimum is the absolute one; in addition, $\varphi$ has a relative maximum $\zeta_2$ such that $-1<\zeta_1<\zeta_2< \zeta_3<1\,.$

Therefore, by \eqref{Cauchytr} we have for the endpoints of the support (zeros of $A$) and the zeros of $B$ that $-1 <a_1(0) = a_2(0) = \zeta_1 <b_1(0) = \zeta_2 <b_2(0) = \zeta_3 <1\,$ in such a way that \eqref{dynamics1} yields: $\dot{a_1}<0, \dot{a_2}>0, \dot{b_1}<0\,$ and $\dot{b_2}>0\,.$ Thus, points $a_2$ and $b_1$ tend to collide, and this collision would take place in a time $T^*<T$ by Theorem \ref{thm:end}; but this would contradict equilibrium condition \eqref{equilibrium}. Hence, there must exist a critical value $T_1<T^*$ where the initial configuration changes. The unique change which takes care of condition \eqref{equilibrium} is the birth of a couple of real zeros $a_3,a_4$ from the rightmost double zero $b_2$. That is, at this critical value $T_1$ a type I singularity occurs and immediately a new cut arises.

Then, for $t>T_1$, combining again Theorem 2.1 and \eqref{equilibrium}, we have that the central endpoints $a_2$ and $a_3$ tend to collide (and, of course, also collide with $b_1$), giving birth to a double root of $B$, that is, a type II singularity. Finally, this double root necessarily splits into a couple of conjugate imaginary roots which tend to some prescribed points in $\C \setminus \R$, while $a_1\rightarrow -\infty$ and $a_4\rightarrow +\infty$, as established in Theorem \ref{thm:end}. In fact, it is not possible that the couple of imaginary roots of $B$ collide giving birth again to a double root of $B$. Indeed, if it were the case, and a new double root were born at, say, $t=T_3>T_2$, using Proposition 4.1, we could vary $\gamma$ and $T_3$, with its corresponding $T_2$, arriving to the collision between the endpoints $a_1$ and $a_2$; thus, the existence of a value of $\gamma$ for which the merger occurs for $t=T_2<T_3=0$ would be established, which is an absurd.

If $\varphi(\zeta_1) = \varphi(\zeta_3)\,,$ which is possible for any $(\beta_1,\beta_2) \in \Omega_0\,$ and a suitable value of $\gamma$, then the evolution is the same as above, but now $T_1=0$ and, thus, the initial one--cut phase does not take place. Obviously, when $\varphi(\zeta_1) > \varphi(\zeta_3)\,,$ the evolution is also the same as above but starting at $\zeta_3$ in place of $\zeta_1$.

The description of the dynamics in scenario (b) is similar, with the unique difference that we start with a couple of conjugate imaginary roots $b_1$ and $b_2 = \overline{b}_1$ which at a certain time $T_0<T^*$ collide at the real axis, becoming a double real root of $B$, in such a way that this double root immediately produces a pair of simple real roots $b_1<b_2$, as at the beginning of the dynamics above.

Conversely, in scenario (c), the initial couple of conjugate imaginary roots $b_1$ and $b_2 = \overline{b}_1$ never attain the real axis and, consequently, the support always consists of a single interval.

Finally, observe that the boundary between scenarios (b) and (c) occurs precisely when the pair of roots of $B$ collide with the root of $A$, producing a type III singularity, which has been studied in previous Theorem 3.3

\section*{Appendix: The ``two--cut'' body}

Throughout this appendix, along with the three--dimension body
\begin{equation}\label{body}
\widetilde{\Delta}\, =\,\{(\beta_1,\beta_2,\gamma) \in (\R^+)^3\,:\,(\beta_1,\beta_2)\in \Omega_0\,,\,\gamma \in (\widetilde{\Gamma}_1(\beta_1,\beta_2),\widetilde{\Gamma}_2(\beta_1,\beta_2))\}\,,
\end{equation}
where the external field $\varphi$ has two minima, we consider the larger body, strictly containing the former one, given by
\begin{equation}\label{mainbody}
\Delta\,=\,\{(\beta_1,\beta_2,\gamma) \in (\R^+)^3\,:\,(\beta_1,\beta_2)\in \Omega_0\,,\,\gamma \in (\Gamma_1(\beta_1,\beta_2),\Gamma_2(\beta_1,\beta_2))\}\,,
\end{equation}
for which range of parameters a two--cut phase takes place.

In Theorems 3.2-3.3 above, the geometry of the problem has been mainly depicted in terms of the admissible values of the heights $(\beta_1,\beta_2)$ in order to guarantee the existence of a two-cut phase for a certain range of the other parameter, the charge $\gamma$. However, next a description of the three-dimension ``two--cut'' body $\Delta$ will be provided to get a better knowledge of the solution of the problem. As shown in Theorem 3.3, it is a body in the first (positive) octant of the $(\beta_1,\beta_2, \gamma)$--space whose projection in the $(\beta_1,\beta_2)$--plane is given by the $\Omega_0$ region in Figure 2. Theorems 3.2 also shows that the projection of both $\widetilde{\Delta}$ on the $(\beta_1,\beta_2)$--plane is given by the same region $\Omega_0$. On the largest body, the existence of a two-cut phase is provided; the unique difference between the phase diagram corresponding to $\widetilde{\Delta}$ and $\Delta\setminus \widetilde{\Delta}$ lies on the fact that for the range of parameters belonging to $\widetilde{\Delta}$, the ``life'' of the two-cut phase is longer.

In this section we restrict ourselves to show the main characteristics of these admissible bodies, especially of $\Delta$, defined in \eqref{mainbody}. This is a three-dimension body bounded by two surfaces (``top and lower covers'').

In this sense, some results will be presented (without proof) and some graphics will be displayed. First, the following properties holds for the three--dimension body $\Delta$
\begin{itemize}

\item The intersection of both surfaces for $(\beta_1,\beta_2)\in \mathcal{C}$ is given by the curve $$\gamma=-\frac{3\beta_1^2+3\beta_2^2-4}{2(3\beta_1^2-4)}+
    \frac{1}{2}\sqrt{\left(\frac{3\beta_1^2+3\beta_2^2-4}{(3\beta_1^2-4)}\right)^2-
    4\frac{3\beta_2^2-4}{3\beta_1^2-4}}$$

\item The intersection with the plane $\beta_2 = 0$ is the whole quadrant $(\R^+)^2$. The same occurs with respect to the plane $\beta_1 = 0$.

\item The following limits hold:
\begin{align*}
&\lim_{\beta_1\searrow 0}\Gamma_1=0\,,\qquad \lim_{\beta_1\searrow 0}\Gamma_2=+\infty\,,\\
&\lim_{\beta_2\searrow 0}\Gamma_1=0\,,\qquad \lim_{\beta_2\searrow 0}\Gamma_2=+\infty\,.
\end{align*}

\end{itemize}

In Figures \ref{fig:secbeta1}-\ref{fig:secdiag} different sections of $\Delta$ are shown.

\begin{figure}
    \begin{center}
        \includegraphics[scale=0.7]{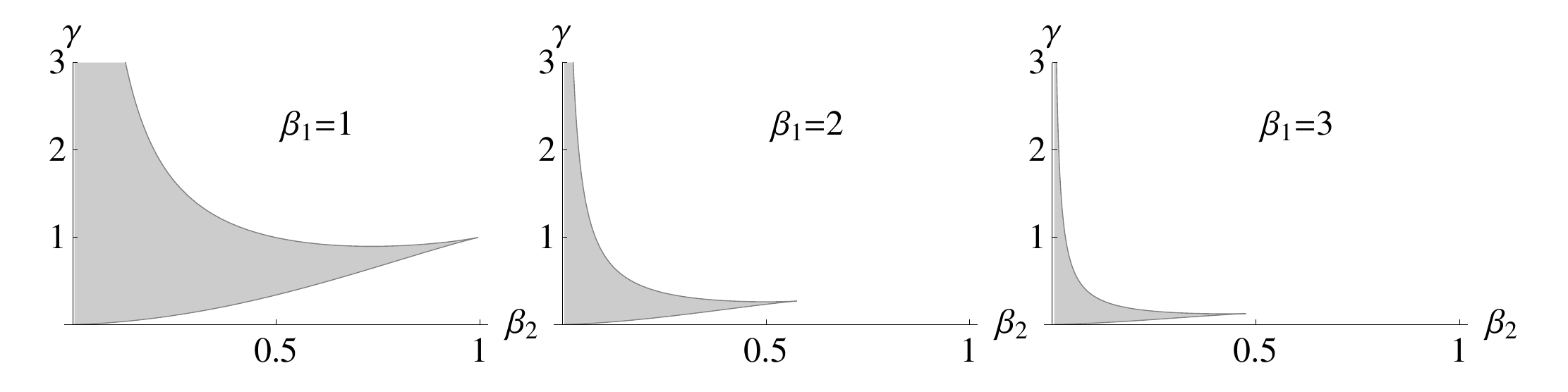}
    \end{center}
    \caption{Vertical sections of $\Delta$ for $\beta_1 = 1,2,3$. }
    \label{fig:secbeta1}
\end{figure}

\begin{figure}
    \begin{center}
        \includegraphics[scale=0.7]{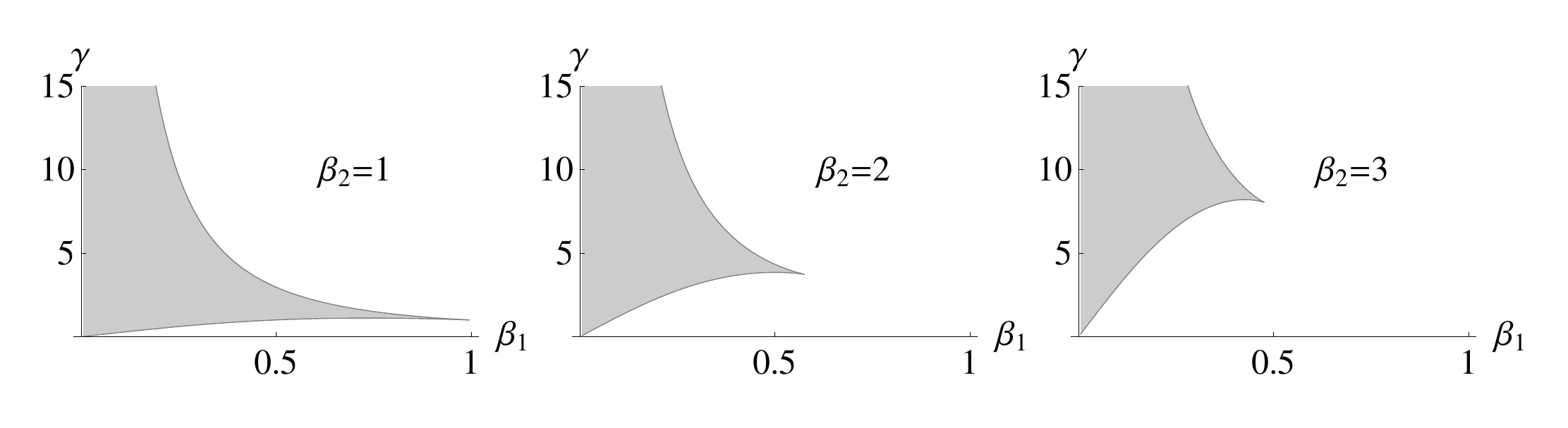}
    \end{center}
    \caption{Vertical sections of $\Delta$ for $\beta_2 = 1,2,3$. }
    \label{fig:secbeta2}
\end{figure}

\begin{figure}
    \begin{center}
        \includegraphics[scale=1.0]{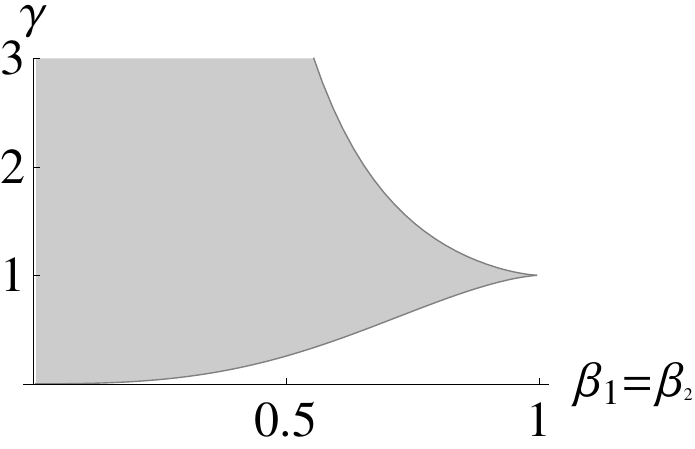}
    \end{center}
    \caption{Section of $\Delta$ for $\beta_1 = \beta_2$. }
    \label{fig:secdiag}
\end{figure}

Finally, with respect to the body $\widetilde{\Delta}$, in \eqref{body}, the following are its main features:
$$\lim_{\beta_2\searrow 0}\widetilde{\Gamma}_1 =0\,,\qquad \lim_{\beta_2\searrow 0}\widetilde{\Gamma}_2 = \frac{-1}{2}+\sqrt{\left(\frac{1}{2}\right)^2+\frac{1}{\beta_1^2}}\,.$$
$$\lim_{\beta_1\searrow 0}\widetilde{\Gamma}_1 =\frac{\beta_2^2}{2}+\frac{\beta_2\sqrt{4+\beta_2^2}}{2}\,,\qquad \lim_{\beta_1\searrow 0}\widetilde{\Gamma}_2=+\infty\,.$$

\obeylines
\texttt{

Ram\'{o}n Orive,

Universidad de La Laguna, Canary Islands, Spain.

rorive@ull.es.

Research partially supported by Ministerio de Ciencia e Innovaci\'{o}n under grants MTM2011-28781 and MTM2015-71352.

\vspace{.5cm}

Joaqu\'{i}n F. S\'{a}nchez Lara,

Universidad de Granada, Spain.

jslara@ugr.es.

Research partially supported by Ministerio de Ciencia e Innovaci\'{o}n under grant MTM2015-71352 and by research project of Junta de Andaluc\'{\i}a under grant FQM384.
}


\newpage

\begin{thebibliography}{10}

\bibitem{AMM2010}
{\'{A}}lvarez, G., Mart{\'{\i}}nez~Alonso, L., Medina, E.:
\newblock Phase transitions in multi-cut matrix models and matched solutions of
  {W}hitham hierarchies.
\newblock J. Stat. Mech. Theory Exp. (3) P03023, 38 (2010).

\bibitem{AlMAMR2015} {\'{A}}lvarez, G., Mart{\'{\i}}nez~Alonso, L., Medina, E.: \newblock Fine structure in the large $n$ limit of the
non-hermitian Penner matrix model. Ann. Physics 361, 440-460 (2015).

\bibitem{AMMT2014}
Atia, M. J., Martinez-Finkelshtein, A., Martinez-Gonzalez, P., Thabet, F.: Quadratic differentials and asymptotics of Laguerre polynomials with varying complex parameters.
\newblock J. Math. Anal. Appl. 416, 52-80 (2014).

\bibitem{BDD2006}
Benko, D., Damelin, S. B., Dragnev, P. D.:  \newblock On the support of the equilibrium measure for arcs of the unit circle and for real intervals.
\newblock Electron. Trans. Numer. Anal. 25, 27–40  (2006).

\bibitem{BlEy2003}
Bleher, P., Eynard, B.:
\newblock Double scaling limit in random matrix models and a nonlinear
  hierarchy of differential equations.
\newblock J. Phys. A, 36(12), 3085--3105 (2003).

\bibitem{Bleher99}
Bleher, P., Its, A. R.:
\newblock Semiclassical asymptotics of orthogonal polynomials,
  {R}iemann-{H}ilbert problems, and universality in the matrix model.
\newblock Ann. of Math. 150, 185--266 (1999).

\bibitem{Bleher/Its03}
Bleher, P., Its, A. R.:
\newblock Double scaling limit in the random matrix model: the
  {R}iemann-{H}ilbert approach.
\newblock Comm. Pure Appl. Math. 56(4), 433--516 (2003).

\bibitem{BPS95}
Boutet~de Monvel, A., Pastur, L., and Shcherbina, M.:
\newblock On the statistical mechanics approach in the random matrix theory:
  integrated density of states.
\newblock J. Statist. Phys., 79(3-4), 585--611 (1995).


\bibitem{Buyarov/Rakhmanov:99}
Buyarov, V. S., Rakhmanov, E. A.:
\newblock On families of measures that are balanced in the external field on
  the real axis.
\newblock Mat. Sb., 190(6), 11--22 (1999).


\bibitem{MR2001g:42050}
Deift, P., Kriecherbauer, T., McLaughlin, K.~T.-R., Venakides, S., Zhou, X.:
\newblock Uniform asymptotics for polynomials orthogonal with respect to
  varying exponential weights and applications to universality questions in
  random matrix theory.
\newblock Comm. Pure Appl. Math., 52(11), 1335--1425 (1999).


\bibitem{DMOr2011} D\'{\i}az Mendoza, C., Orive, R.:
\newblock The Szeg\"{o} curve and Laguerre polynomials with large negative parameters.
\newblock J. Math. Anal. Appl. 379, 305-315 (2011).

\bibitem{DiVA 00} Dimitrov, D.~K., Van Assche, V.: Lam\'{e}
differential equations and electrostatics, Proc. Amer. Math. Soc.
128, 3621--3628 (2000).

\bibitem{Todastrings}
Dijkgraaf, R., Vafa, C.: \newblock Toda Theories, Matrix Models, Topological Strings, and N=2
Gauge Systems. \newblock preprint arXiv:0909.2453.

\bibitem{Dubrovin:1975fk}
Dubrovin, B.:
\newblock Periodic problems for the {K}orteweg-de {V}ries equation in the class
  of finite band potentials.
\newblock Funct. Anal. Appl., 9, 215--223 (1975).

\bibitem{Eguchi}
Eguchi, T., Maruyoshi, K.: \newblock Penner Type Matrix Model and Seiberg-Witten Theory.
\newblock J. High Energy Phys.,2, 022, 21 pp. (2010).

\bibitem{Gr 02} Grinshpan, A.: A minimum energy problem
and Dirichlet spaces, Proc. Amer. Math. Soc. 130, no. 2,
453--460 (2002).

\bibitem{Grum 98} Gr\H{u}nbaum, F. A.: Variations on a theme
of Heine and Stieltjes: an electrostatic interpretation of the
zeros of certain polynomials. J. Comput. Appl. Math. 99,
189--194 (1998).

\bibitem{Grum 01} Gr\H{u}nbaum, F. A.: Electrostatic interpretation for the
zeros of certain polynomials and the Darboux process. J. Comput.
Appl. Math. 133, 397--412 (2001).

\bibitem{He 1878} Heine, E.: Handbuch der kugelfunctionen,
vol. II, G. Reimer, Berlin, 2nd. edition (1878).

\bibitem{Johansson}
Johansson, K.:
\newblock Random matrices and determinantal processes.
\newblock In Mathematical statistical physics, 1--55. Elsevier B.
  V., Amsterdam (2006).

\bibitem{Kris2006}
Krishnaswami, G. S.: \newblock Phase transition in matrix model with logarithmic action: toy-model for gluons in baryons.
J. High Energy Phys. no. 3, 067, 22 pp. (electronic) (2006).

\bibitem{KuMF2004}
Kuijlaars, A. B. J., Martinez-Finkelshtein, A.: \newblock Strong asymptotics for Jacobi polynomials with varying nonstandard parameters,
\newblock J. d'Analyse Mathematique 94, 195-234 (2004).

\bibitem{KuML 00}
Kuijlaars, A. B. J., McLaughlin, K.~T.-R.:
\newblock Generic behavior of the density of states in random matrix theory and
  equilibrium problems in the presence of real analytic external fields.
\newblock Comm. Pure Appl. Math., 53(6), 736--785 (2000).

\bibitem{KuML2001}
Kuijlaars, A. B. J., McLaughlin, K.~T.-R.: Riemann-Hilbert
analysis for Laguerre polynomials with large negative parameter.
Comput. Met. Funct. Theory 1, 205--233 (2001).

\bibitem{KuML2004}
Kuijlaars, A. B. J., McLaughlin, K.~T.-R.: Asymptotic zero
behavior of Laguerre polynomials with negative parameter. 
Constr. Approx. 20, 497--523 (2004).

\bibitem{MaMFMG 07} F.~Marcell\'{a}n,
Mart\'{\i}nez-Finkelshtein, A., Mart\'{\i}nez-Gonz\'{a}lez, P.:
Electrostatic models for zeros of polynomials: Old, new and some
open problems. J. Comp. Appl. Math. 207, 258--272 (2007).

\bibitem{MFMGOr2000}
Mart\'{\i}nez-Finkelshtein, A., Mart\'{\i}nez-Gonz\'{a}lez, P., Orive, R.:  Zeros of {J}acobi polynomials with varying
non-classical parameters. In Special functions (Hong Kong,
1999), World Sci. Publishing, River Edge, NJ, , 98--113 (2000).

\bibitem{MFMGOr2001}
Mart\'{\i}nez-Finkelshtein, A., Mart\'{\i}nez-Gonz\'{a}lez, P., Orive, R.: On asymptotic zero
distribution of {L}aguerre and generalized {B}essel polynomials with varying
parameters. J. Comput. Appl. Math. 133, 477--487 (2001).

\bibitem{MFOr2005}
Mart\'{\i}nez-Finkelshtein, A., Orive, R.: Riemann-Hilbert
analysis for Jacobi polynomials orthogonal on a single contour.
J. Approx. Theory 134, 137--170 (2005).

\bibitem{MFMGOr 05}
Mart\'{\i}nez-Finkelshtein, A., Mart\'{\i}nez-Gonz\'{a}lez, P., Orive, R.: \newblock
Asymptotics of polynomial solutions of a class of generalized Lam\'{e}
differential equations. \newblock Elect. Trans. in Numer. Anal. 19,
18-28 (2005).

\bibitem{MOR2015}
Mart{\'{\i}}nez-Finkelshtein, A., Orive, R., Rakhmanov, E. A.: \newblock Phase transitions and equilibrium measures in random matrix models. \newblock Comm. Math. Phys. 333, 1109–1173 (2015).

\bibitem{MFRa2011}
Mart{\'{\i}}nez-Finkelshtein, A., Rakhmanov, E. A.:
\newblock Critical measures, quadratic differentials, and weak limits of zeros
  of {S}tieltjes polynomials.
\newblock Comm. Math. Phys., 302(1), 53--111 (2011).

\bibitem{MFSa2002}
Mart{\'{\i}}nez-Finkelshtein, A., Saff, E. B.:
\newblock Asymptotic properties of {H}eine-{S}tieltjes and {V}an {V}leck
  polynomials.
\newblock J. Approx. Theory, 118(1), 131--151 (2002).

\bibitem{Mehta2004}
Mehta, M. L.:
\newblock Random Matrices, volume 142 of Pure and Applied
  Mathematics.
\newblock Academic Press, 3rd edition (2004).

\bibitem{OrGa2010}
Orive, R., Garc{\'{\i}}a, Z.:
\newblock On a class of equilibrium problems in the real axis.
\newblock J. Comput. Appl. Math. 235(4), 1065--1076 (2010).

\bibitem{OrSL2015}
Orive, R., S\'{a}nchez--Lara, J. F.:
\newblock Equilibrium measures in the presence of certain rational external fields.
\newblock  J. Math. Anal. Appl. 431, 1224-1252 (2015).


\bibitem{Ron 95} Ronveaux, A. (ed.): Heun's differential
equations, Oxford Science Publications (1995).

\bibitem{Saff:97}
Saff, E. B., Totik, V.:
\newblock Logarithmic Potentials with External Fields, volume 316 of
  Grundlehren der Mathematischen Wissenschaften.
\newblock Springer-Verlag, Berlin (1997).

\bibitem{Stahl:86}
Stahl, H.:
\newblock Orthogonal polynomials with complex-valued weight function. {I},
  {II}.
\newblock Constr. Approx., 2(3), 225--240, 241--251 (1986).

\bibitem{St 1885} Stieltjes, T. J.: Sur certain
polynomes que verifient une equation differentielle lineaire du
second ordre et sur la teorie des fonctions de Lam\'e. Acta Math.
6, 321--326 (1885).

\bibitem{St 1885b} Stieltjes, T. J.: Sur quelques th\'{e}or\`{e}mes
d'alg\`{e}bre. Comptes Rendus de l'Acad\'{e}mie des Sciences, Paris 100, 439--440 (1885).

\bibitem{St 1885c} Stieltjes, T. J.: Sur les polyn\'{o}mes de Jacobi. Comptes Rendus de l'Acad\'{e}mie des Sciences, Paris 100, 620--622 (1885).

\bibitem{St 1885d} Stieltjes, T. J.: Sur les racines de l'\'{e}quation $X_n=0$. Acta Math.
9, 385--400 (1886).

\bibitem{Sz 75}
Szeg\H{o}, G.: Orthogonal Polynomials, volume~23 of Amer.\
Math.\ Soc.\ Colloq.\ Publ., Amer.\ Math.\ Soc., Providence, {RI},
fourth edition (1975).

\bibitem{Wangbook}
Wang, C. B.:
\newblock Application of integrable systems to phase transitions,
\newblock Springer, Heidelberg (2013).

\end{thebibliography}
\end{document}